\let\ORIlabel\label
\let\ORIrefstepcounter\refstepcounter
\AddToHook{package/hyperref/before}{
   \let\label\ORIlabel 
   \let\refstepcounter\ORIrefstepcounter}

\documentclass{siamart220329}
\usepackage{lipsum}
\usepackage{amsfonts}
\usepackage{graphicx}
\usepackage{epstopdf}
\usepackage{algorithmic}
\usepackage{dsfont}
\usepackage{enumitem}
\usepackage{amsmath,amssymb,bm,dsfont}
\usepackage{pdflscape}
\usepackage{subcaption}
\usepackage{tabularx} 
\setlist[description]{style=unboxed,leftmargin=.5em}
\setlist[itemize]{style=sameline,leftmargin=2em}

\ifpdf  \DeclareGraphicsExtensions{.eps,.pdf,.png,.jpg}
\else
  \DeclareGraphicsExtensions{.eps}
\fi

\newsiamremark{remark}{Remark}
\newsiamremark{hypothesis}{Hypothesis}
\crefname{hypothesis}{Hypothesis}{Hypotheses}
\newsiamthm{claim}{Claim}
\newsiamthm{assumption}{Assumption}

\def\[{\begin{equation*}}
\def\]{\end{equation*}}

\def\cK{{\cal K}}

\def\cM{{\cal M}}
\def\cC{{\cal C}}
\def\cT{{\cal T}}

\def\cP{{\cal P}}

\def\hw{{\hat w}}
\def\by{{\bar y}}
\def\bz{{\bar z}}

\def\bx{{\bar x}}

\def\bw{{\bar w}}

\title{On the Relationships among GPU-Accelerated  First-Order Methods for Solving Linear Programming\thanks{September 30, 2025 \funding{The work of Defeng Sun was supported by the Research Center for Intelligent Operations Research, RGC {Senior Research Fellow Scheme No. SRFS2223-5S02}, and  {GRF Project No. 15307822}. The work of Yancheng Yuan was supported by the RGC Early Career Scheme (Project No.
25305424) and the Research Center for Intelligent Operations Research. The work of Xinyuan Zhao was supported in part by the National Natural Science Foundation of China under Project { No. 12271015}.
}
}}
\author{Kaihuang Chen\thanks{Department of Applied Mathematics, The Hong Kong Polytechnic University, Hung Hom, Hong Kong,
  (\email{kaihuang.chen@connect.polyu.hk}).}
\and Defeng Sun\thanks{Department of Applied Mathematics, The Hong Kong Polytechnic University, Hung Hom, Hong Kong,
(\email{defeng.sun@polyu.edu.hk}).} \and 
Yancheng Yuan\thanks{Department of Applied Mathematics, The Hong Kong Polytechnic University, Hung Hom, Hong Kong,
(\email{yancheng.yuan@polyu.edu.hk}).}
\and Guojun Zhang\thanks{Department of Applied Mathematics, The Hong Kong Polytechnic University, Hung Hom, Hong Kong,
(\email{guojun.zhang@connect.polyu.hk}).}
\and Xinyuan Zhao\thanks{Department of Mathematics, Beijing University of Technology, Beijing, P.R. China,
(\email{xyzhao@bjut.edu.cn}).}
}

\begin{document}

\maketitle

\begin{abstract}
This paper aims to understand the relationships among recently developed GPU-accelerated first-order methods (FOMs) for linear programming (LP), with particular emphasis on HPR-LP---a Halpern Peaceman--Rachford (HPR) method for LP. Our findings can be summarized as follows: (i) the base algorithm of cuPDLPx, a recently released GPU solver, is a special case of the base algorithm of HPR-LP, thereby showing that cuPDLPx is another concrete implementation instance of HPR-LP; (ii) once the active sets have been identified, HPR-LP and EPR-LP---an ergodic PR method for LP---become equivalent under the same initialization; and (iii) extensive numerical experiments on benchmark datasets demonstrate that HPR-LP achieves the best overall performance among current GPU-accelerated LP solvers. These findings provide a strong motivation for using the HPR method as a baseline to further develop GPU-accelerated LP solvers and beyond.
\end{abstract}



\section{Introduction}\label{sec:intro}
In this paper, we focus on GPU-accelerated first-order methods (FOMs) for solving large-scale linear programming (LP) problems. Our aim is to develop a deep understanding of the connections among recent algorithmic developments and to provide insights that will guide the design of future high-performance LP solvers. Specifically, we consider the following general form of LP:
\begin{equation}\label{model:primalLP}
\begin{aligned}
\min _{x \in \mathbb{R}^n} &\quad  \langle c, x \rangle \\
\text { s.t. } & Ax\in \mathcal{K},\\
& x \in \mathcal{C},
\end{aligned}
\end{equation}
where  $c \in \mathbb{R}^n$ is the objective vector, $A \in \mathbb{R}^{m \times n}$ is the constraint matrix, $\mathcal{K} := \{s \in \mathbb{R}^m : l_c \leq s \leq u_c\}$ with bounds 
$l_c \in (\mathbb{R} \cup \{-\infty\})^m$ and $u_c \in (\mathbb{R} \cup \{\infty\})^m$, and $\mathcal{C} := \{x \in \mathbb{R}^n : l_v \leq x \leq u_v\}$ with bounds 
$l_v \in (\mathbb{R} \cup \{-\infty\})^n$ and $u_v \in (\mathbb{R} \cup \{\infty\})^n$.
 The corresponding dual problem is:
\begin{equation}\label{model:dualLP}
\begin{aligned}
\min_{y \in \mathbb{R}^m,\; z \in \mathbb{R}^n} \quad & \delta_{\mathcal{K}}^*(-y) + \delta_{\mathcal{C}}^*(-z) \\
\text{s.t.} \quad & A^{*} y + z = c,
\end{aligned}
\end{equation}
where $\delta_S^*(\cdot)$ denotes the convex conjugate of the indicator function $\delta_S(\cdot)$ associated with a closed convex set $S$.

Classical LP algorithms such as the simplex method~\cite{dantzig1963linear} and interior-point methods~\cite{karmarkar1984new} relied on factorization-based approaches. Commercial solvers like CPLEX~\cite{cplex2009v12} and Gurobi~\cite{gurobi} combined simplex and interior-point methods with highly optimized sparse linear algebra, delivering state-of-the-art CPU performance. Yet their reliance on inherently sequential matrix factorizations makes it difficult to fully exploit massively parallel GPUs with high memory bandwidth~\cite{lu2025overview}. Even GPU-accelerated interior-point solvers like CuClarabel~\cite{goulart2024clarabel,chen2024cuclarabel}, which leverage mixed precision, still depend on direct factorizations and thus face scalability challenges on large-scale problems.

In recent years, first-order methods (FOMs) have attracted growing attention for large-scale LPs due to their low per-iteration cost and high parallelizability. Representative solvers include PDLP~\cite{applegate2021practical,applegate2023faster}, ECLIPSE~\cite{basu2020eclipse}, ABIP~\cite{lin2021admm,deng2024enhanced}, and HPR-LP~\cite{chen2024hpr}, with related convex conic or quadratic solvers such as SCS~\cite{o2016conic,o2021operator}, OSQP~\cite{stellato2020osqp},  PDQP~\cite{lu2025practical}, PDHCG~\cite{huang2024restarted}, PDCS~\cite{lin2025pdcs}, and HPR-QP~\cite{chen2025hpr}. A milestone in this line is PDLP~\cite{applegate2021practical,applegate2025pdlp}, based on the primal dual hybrid gradient (PDHG) method~\cite{zhu2008efficient,esser2010general,chambolle2011first}, which is equivalent to linearized ADMM with unit stepsize (or linearized Douglas--Rachford)~\cite{esser2010general,chambolle2011first}. Its main computational advantage is that it relies only on sparse matrix-vector multiplications and simple projections, avoiding matrix factorizations. Practical enhancements such as restarts with ergodic sequences, adaptive penalty updates, and heuristic line search further improve robustness and efficiency. Notably, PDLP received the 2024 Beale--Orchard-Hays Prize for Excellence in Computational Mathematical Programming. Its GPU implementations, including cuPDLP.jl~\cite{lu2023cupdlp}, cuPDLP-C~\cite{lu2023cupdlp_c}, and NVIDIA’s cuOpt\footnote{\url{https://developer.nvidia.com/blog/accelerate-large-linear-programming-problems-with-nvidia-cuopt}}, have demonstrated competitive performance against commercial solvers on large-scale LP benchmarks. Furthermore, instead of relying on costly heuristic line search as in PDLP, Chen et al.~\cite{chen2025peaceman} established the ergodic convergence of the semi-proximal Peaceman--Rachford (PR) method, which admits larger effective stepsizes at low cost. Building on this foundation, a preliminary implementation of ergodic PR for LP with practical enhancements, termed EPR-LP, has already demonstrated superior empirical performance over cuPDLP.jl~\cite{zhangergodic}.

From a theoretical standpoint, ergodic sequences of semi-proximal ADMM, including PDHG, achieve an $O(1/k)$ iteration complexity for objective error and feasibility violation~\cite{davis2016convergence,cui2016convergence}. Recently, building on the ergodic $O(1/k)$ complexity of PDHG established by Chambolle and Pock~\cite{chambolle2011first,chambolle2016ergodic}, Applegate et al.~\cite{applegate2023faster} proved that PDHG achieves the ergodic $O(1/k)$ complexity for LP problems when measured by primal-dual feasibility violation and the primal-dual gap, both of which can be inferred from the Karush--Kuhn--Tucker (KKT) residual. For the KKT residual itself, Monteiro and Svaiter~\cite{monteiro2013iteration} established an ergodic $O(1/k)$ rate for ADMM 
(with unit dual stepsize) in terms of the $\varepsilon$-subdifferential. This result was later extended to sPADMM by Shen and Pan~\cite{shen2016weighted}, and further to the semi-proximal PR method by Chen et al.~\cite{chen2025peaceman}, implying an ergodic $O(1/\sqrt{k})$ complexity for the KKT residual.

In contrast, HPR-LP~\cite{chen2024hpr}, a Halpern Peaceman--Rachford (HPR) method for LP, achieves a nonergodic $O(1/k)$ rate directly in terms of the KKT residual. It is derived by accelerating the preconditioned (semi-proximal) PR method~\cite{zhang2022efficient,yang2025accelerated,sun2025accelerating} through the Halpern iteration~\cite{halpern1967fixed,lieder2021convergence,sabach2017first}. Specifically, Zhang et al.~\cite{zhang2022efficient} applied the Halpern iteration to the classical PR method~\cite{eckstein1992douglas,lions1979splitting}, obtaining the HPR method without proximal terms and establishing an $O(1/k)$ rate for both the KKT residual and the objective error. Subsequently, Sun et al.~\cite{sun2025accelerating} reformulated the semi-proximal PR method as a degenerate proximal point method (dPPM)~\cite{bredies2022degenerate} with a positive semidefinite preconditioner, and applied the Halpern iteration to obtain the HPR method with semi-proximal terms, which also achieves an $O(1/k)$ rate. Building on these advances, Chen et al.~\cite{chen2024hpr} developed HPR-LP, a GPU-accelerated solver for large-scale LP that significantly outperformed the award-winning PDLP~\cite{applegate2021practical,applegate2023faster,lu2023cupdlp}. Inspired by the work of HPR-LP, Lu and Yang~\cite{lu2024restarted} proposed the reflected restarted Halpern PDHG ($r^2$HPDHG), a special case of the HPR method~\cite{sun2025accelerating}. Building on this, Lu et al.~\cite{lu2025cupdlpx} released an efficient GPU implementation, cuPDLPx, which exhibits superior performance compared to cuPDLP.jl \cite{lu2023cupdlp}.

Despite significant progress in GPU-accelerated FOMs for LP, their interrelationships remain unclear. This paper addresses this gap by investigating the connections among recent approaches, with a focus on HPR-LP and its links to other solvers. The main contributions can be summarized as follows:

\begin{enumerate}
    \item  We prove that the base algorithm of cuPDLPx ($r^2$HPDHG) is a special case of the base algorithm of HPR-LP, implying that cuPDLPx is another concrete implementation instance of HPR-LP.
    \item  We show that, under the strict complementarity condition, the PR update in the HPR method for LP becomes affine after finitely many iterations. Consequently, after active sets are identified, HPR-LP and EPR-LP are equivalent under consistent initialization based on the equivalence between Halpern iteration and ergodic iteration for affine fixed-point maps. 
    \item We present extensive experiments on LP benchmark datasets, evaluating the impact of algorithmic enhancements and demonstrating that HPR-LP achieves the best overall performance among current GPU-accelerated LP solvers.
\end{enumerate}

The remainder of this paper is organized as follows. Section~\ref{sec:2} briefly reviews the HPR method for solving LP. Section~\ref{sec:3} discusses the relationship between cuPDLPx and HPR-LP as well as the connection between HPR-LP and EPR-LP. Section~\ref{sec:4} presents extensive numerical results on recent GPU-accelerated FOMs for LP across different benchmark sets. Finally, Section~\ref{sec:5} concludes the paper.

\paragraph{Notation} Let \(\mathbb{R}^n\) be the \(n\)-dimensional Euclidean space with inner product \(\langle \cdot, \cdot \rangle\) and norm \(\|\cdot\|\). For \(A \in \mathbb{R}^{m \times n}\), let \(A^{*}\) be its transpose and \(\|A\| := \sqrt{\lambda_{1}(AA^{*})}\) its spectral norm, where \(\lambda_{1}(\cdot)\) denotes the largest eigenvalue. For any self-adjoint positive semidefinite operator \(\mathcal{M}: \mathbb{R}^n \to \mathbb{R}^n\), define the seminorm \(\|x\|_{\cM} := \sqrt{\langle x, \cM x \rangle}\). For a convex function \(f: \mathbb{R}^n \to (-\infty,+\infty]\), let \(\partial f(\cdot)\) be its subdifferential and 
\(\operatorname{Prox}_f(x) := \arg\min_{z \in \mathbb{R}^n} \left\{ f(z) + \tfrac{1}{2}\|z-x\|^2 \right\}\) its proximal mapping. For a convex set \(C \subseteq \mathbb{R}^n\), we use the following notations:  
the indicator function \(\delta_C(x) := 0\) if \(x \in C\) and \(+\infty\) otherwise;  
the distance under \(\|\cdot\|_{\cM}\), \(\operatorname{dist}_{\cM}(x,C) := \inf_{z \in C} \|x-z\|_{\cM}\);  
the Euclidean projection \(\Pi_C(x) := \arg\min_{z \in C} \|x-z\|\);  
and the normal cone at \(x \in C\), denoted by \(\mathcal{N}_C(x)\).

\section{HPR-LP: An HPR method for solving LP}\label{sec:2}
In this section,  we begin by presenting the base algorithm of HPR-LP \cite{chen2024hpr}, followed by a discussion of its convergence guarantees and complexity results, which motivate subsequent algorithmic enhancements.  

\subsection{Base algorithm}
For any $(y,z,x)\in \mathbb{R}^{m} \times \mathbb{R}^{n} \times \mathbb{R}^{n}$, the augmented Lagrangian of the dual problem~\eqref{model:dualLP} is
$$
L_{\sigma}(y,z;x):=\delta_{\mathcal{K}}^*(-y) + \delta_{\mathcal{C}}^{*}(-z)+\langle x, A^{*}y+z-c \rangle +\frac{\sigma}{2}\|A^{*}y+z-c\|^2,
$$
where $\sigma>0$ is a penalty parameter. For notational convenience, let $w:=(y,z,x)\in \mathbb{W}:= \mathbb{R}^{m} \times \mathbb{R}^{n} \times \mathbb{R}^{n}$. Then, an HPR method with semi-proximal terms \cite{sun2025accelerating,chen2024hpr} for solving problems \eqref{model:primalLP} and \eqref{model:dualLP} is summarized in Algorithm~\ref{alg:sp-HPR}.  
\begin{algorithm}[ht!]
		\caption{An HPR method  with semi-proximal terms for the problem \eqref{model:dualLP}}
		\label{alg:sp-HPR}
		\begin{algorithmic}[1]
			\STATE {\textbf{Input:} Set the penalty parameter \(\sigma > 0\). Let \(\mathcal{T}_1: \mathbb{R}^m \to \mathbb{R}^m\) be a self-adjoint positive semidefinite linear operator such that \(\mathcal{T}_1 + AA^{*}\) is positive definite. Denote $w=(y,z,x)$ and $\bw=(\by,\bz,\bx)$.
            Choose an initial point \(w^0 = (y^0, z^0, x^0) \in \mathbb{R}^{m}  \times \mathbb{R}^n \times \mathbb{R}^n\).\vspace{3pt}}
               \FOR{$k=0,1,...,$ \vspace{3pt}}
			\STATE {Step 1. $\displaystyle \bz^{k+1}=\underset{z \in \mathbb{R}^{n}}{\arg \min }\left\{L_{\sigma}\left(y^k, z ; x^k\right)\right\}$;}
			\STATE{Step 2. $ \displaystyle \bx^{k+1}={x}^k+\sigma (A^{*}{y}^{k}+\bz^{k+1}-c) $;}
			\STATE {Step 3. $\displaystyle \by^{k+1}=\underset{y \in \mathbb{R}^{m}}{\arg \min }\left\{L_{\sigma}\left(y, \bz^{k+1} ; \bx^{k+1}\right)+\frac{\sigma}{2}\|y-y^{k}\|_{\mathcal{T}_1}^2\right \}$;}
			\STATE {Step 4. $\displaystyle \hw^{k+1}= 2\bw^{k+1}-{w}^{k} $;}

			\STATE {Step 5.  $\displaystyle w^{k+1}=\frac{1}{k+2} w^{0}+\frac{k+1}{k+2}\hw^{k+1}$; \vspace{3pt}}
               \ENDFOR \vspace{3pt}
             \STATE{\textbf{Output:} Iteration sequence $\{\bw^k\}$.}
		\end{algorithmic}
	\end{algorithm}
\begin{remark}
Steps 1–3 correspond to the semi-proximal DR method~\cite{glowinski1975approximation,gabay1976dual}. Adding Step 4 (relaxation) yields the semi-proximal PR method, and Step 5 introduces Halpern iteration with stepsize $1/(k+2)$~\cite{halpern1967fixed,lieder2021convergence}. Together, Algorithm~\ref{alg:sp-HPR} is an accelerated preconditioned ADMM (pADMM) with parameters $\alpha=2$ and $\rho=2$~\cite{sun2025accelerating}.
\end{remark}

According to \cite[Corollary 28.3.1]{rockafellar1970convex}, a pair $({y}^*, {z}^*) \in \mathbb{R}^{m} \times \mathbb{R}^{n}$ is an optimal solution to problem~\eqref{model:dualLP} if there exists ${x}^* \in \mathbb{R}^{n}$ such that $({y}^*, {z}^*, {x}^*)$ satisfies the following KKT system:
\begin{equation}\label{eq:KKT}
	0 \in Ax^* - \partial \delta^{*}_{\cK}(-y^*), 
	\quad  
	0 \in x^* - \partial \delta^{*}_{\cC}(-z^*), 
	\quad  
	A^{*}{y}^* + {z}^* - c = 0.
\end{equation}
We make the following assumption:
\begin{assumption}\label{ass:CQ}
There exists a vector $({y}^*, {z}^*, {x}^*) \in \mathbb{R}^{m} \times \mathbb{R}^{n} \times \mathbb{R}^{n}$ satisfying the KKT system~\eqref{eq:KKT}.
\end{assumption}
Under Assumption~\ref{ass:CQ}, solving the primal--dual pair \eqref{model:primalLP}--\eqref{model:dualLP} is equivalent to finding a point $w^{*} \in \mathbb{R}^{m} \times \mathbb{R}^{n} \times \mathbb{R}^{n}$ such that $0 \in \cT w^{*}$, where the maximal monotone operator $\cT$ is defined as
\begin{equation}\label{def:T}
	\cT w =
	\begin{pmatrix}
		-\partial \delta^{*}_{\cK}(-y) + A x \\
		-\partial \delta^{*}_{\cC}(-z) + x \\
		c - A^{*} y - z
	\end{pmatrix}
	\quad \forall w=(y,z,x) \in \mathbb{R}^{m} \times \mathbb{R}^{n} \times \mathbb{R}^{n}.
\end{equation}
The global convergence of Algorithm~\ref{alg:sp-HPR} is established in the following proposition.

\begin{proposition}[Corollary 3.5 in \cite{sun2025accelerating}]
Suppose that Assumption~\ref{ass:CQ} holds. Then the sequence 
\(\{\bw^k\} = \{(\by^k, \bz^k, \bx^k)\}\) generated by the HPR method with semi-proximal terms in Algorithm~\ref{alg:sp-HPR} converges to a point \(w^{*} = (y^*, z^*, x^*)\), where \((y^*, z^*)\) solves the dual problem~\eqref{model:dualLP} and \(x^*\) solves the primal problem~\eqref{model:primalLP}.
\end{proposition}

Next, consider the self-adjoint positive semidefinite linear operator $\cM: \mathbb{R}^{m} \times \mathbb{R}^{n} \times \mathbb{R}^{n} \to \mathbb{R}^{m} \times \mathbb{R}^{n} \times \mathbb{R}^{n}$ defined by
\begin{equation}\label{def:M}
	\cM =
	\begin{bmatrix}
		\sigma A A^{*} + \sigma \cT_1 & 0 & A \\
		0 & 0 & 0 \\
		A^{*} & 0 & \frac{1}{\sigma} I_{n}
	\end{bmatrix},
\end{equation}
where $I_{n}$ denotes the $n \times n$ identity matrix. To analyze the complexity of the HPR method with semi-proximal terms, we consider the KKT residual and the objective error. The residual mapping associated with the KKT system~\eqref{eq:KKT}, as introduced in \cite{han2018linear}, is given by
\begin{equation}\label{def:KKT_residual}
	\mathcal{R}(w) =
	\begin{pmatrix}
		Ax - \Pi_{\cK}(Ax - y) \\
		x - \Pi_{\cC}(x - z) \\
		c - A^{*} y - z
	\end{pmatrix}
	\quad \forall w=(y,z,x) \in \mathbb{R}^{m} \times \mathbb{R}^{n} \times \mathbb{R}^{n}.
\end{equation}
Furthermore, let $\{(\by^k,\bz^k)\}$ be the sequence generated by Algorithm~\ref{alg:sp-HPR}. We define the objective error as
\[
h(\by^{k+1}, \bz^{k+1}) := \delta_{\cK}^{*}(-\by^{k+1}) + \delta_{\cC}^{*}(-\bz^{k+1})
- \delta_{\cK}^{*}(-y^{*}) - \delta_{\cC}^{*}(-z^*) 
\quad \forall k \geq 0,
\]
where $(y^*, z^*)$ is the limit point of the sequence $\{(\by^k, \bz^k)\}$. The complexity of the HPR method with semi-proximal terms is summarized in Theorem \ref{Th:complexity-acc-pADMM}.
\begin{theorem}[Proposition 2.9 and Theorem 3.7 in \cite{sun2025accelerating}]\label{Th:complexity-acc-pADMM}
Suppose that Assumption~\ref{ass:CQ} holds. Let $\{w^k\}=\{(y^{k},z^{k},x^{k})\}$ and $\{\bw^k\}=\{(\by^{k},\bz^{k},\bx^{k})\}$ be two sequences generated by the HPR method with semi-proximal terms in Algorithm~\ref{alg:sp-HPR}, and let $w^*=(y^*,z^*,x^*)$ be its limit point. Define $R_0=\|w^{0}-w^{*}\|_{\cM}$. Then for all $k \geq 0$, the following iteration complexity bounds hold:
\begin{subequations}\label{eq:complexity-bounds}
\begin{align}
\|\bw^{k+1}-{w}^{k}\|_{\cM} &\leq \frac{R_{0}}{k+1}, \label{eq:complexity-bound-M}\\
\|\mathcal{R}(\bar w^{k+1})\| &\leq 
\left( \frac{\sigma (\|A\|+\|\sqrt{\cT_1}\|)+1}{\sqrt{\sigma}} \right) \frac{R_0}{k+1}, \label{eq:complexity-bound-KKT}\\
\left(-\frac{1}{\sqrt{\sigma}}\|x^*\|\right)\frac{R_0}{k+1} 
&\leq h(\by^{k+1},\bz^{k+1}) 
\leq \left(3R_0 + \frac{1}{\sqrt{\sigma}}\|x^*\|\right)\frac{R_0}{k+1}. \label{eq:complexity-bound-obj}
\end{align}
\end{subequations}
\end{theorem}
The results above establish that the HPR method (Algorithm \ref{alg:sp-HPR}) enjoys an $O(1/k)$ complexity rate in terms of the KKT residual and the objective error. These properties motivate the use of restart strategies and adaptive parameter updates, which will be discussed in the next subsection.

\subsection{Algorithmic enhancements}\label{subsec:algo-enhance}
Several enhancements have been proposed to improve the performance of the HPR method for solving LP \cite{chen2024hpr} and convex composite quadratic programming (CCQP) \cite{chen2025hpr}. In particular, restart strategies and adaptive updates of the penalty parameter \(\sigma\), motivated by the $O(1/k)$ complexity results in Theorem~\ref{Th:complexity-acc-pADMM}, have proven effective. For completeness, we summarize the HPR-LP framework with these enhancements in Algorithm~\ref{alg:HPR_LP}.

\begin{algorithm}[ht!]
		\caption{HPR-LP: A Halpern Peaceman--Rachford method for the LP problem \eqref{model:dualLP} (cf.~\cite{chen2024hpr})}
		\label{alg:HPR_LP}
		\begin{algorithmic}[1]
			\STATE {\textbf{Input:} Let \(\mathcal{T}_1: \mathbb{R}^m \to \mathbb{R}^m\) be a self-adjoint positive semidefinite linear operator such that \(\mathcal{T}_1 +  AA^{*}\) is positive definite. {Denote $w=(y,z,x)$ and $\bw=(\by,\bz,\bx)$.} Choose an initial point \(w^{0,0} = (y^{0,0}, z^{0,0}, x^{0,0}) \in \mathbb{R}^m \times \mathbb{R}^n \times \mathbb{R}^n\).\vspace{3pt} }
                \STATE{\textbf{Initialization}: Set the outer loop counter \(r = 0\), the total loop counter \(k = 0\), and the initial penalty parameter \(\sigma_{0} > 0\). \vspace{3pt}}
                \REPEAT  \vspace{3pt}
                       \STATE{{initialize the inner loop: set inner loop counter $t=0$;} \vspace{3pt}}
                       \REPEAT \vspace{3pt}
			\STATE {$\displaystyle \bz^{r,t+1}=\underset{z \in \mathbb{R}^{n}}{\arg \min }\left\{L_{\sigma_r}\left(y^{r,t}, z ; x^{r,t}\right)\right\}$; \vspace{3pt} }
			\STATE{$\displaystyle \bx^{r,t+1}={x}^{r,t}+{\sigma_r} (A^{*}{y}^{r,t}+\bz^{r,t+1}-c) $; \vspace{3pt}}
			\STATE {$\displaystyle \by^{r,t+1}=\underset{y \in \mathbb{R}^{m}}{\arg \min }\left\{L_{\sigma_r}\left(y, \bz^{r,t+1} ; \bx^{r,t+1}\right)+\frac{\sigma_{r}}{2}\|y-y^{r,t}\|_{\mathcal{T}_1}^2\right \};$ \vspace{3pt}	}
			\STATE {$\displaystyle \hw^{r,t+1}= 2\bw^{r,t+1}-{w}^{r,t} $; \vspace{3pt} }
			\STATE {$\displaystyle w^{r,t+1}=\frac{1}{t+2} w^{r,0}+\frac{t+1}{t+2}\hw^{r,t+1}$; \vspace{3pt}}
   			\STATE {$t=t+1$, $k=k+1$;}
                 \UNTIL one of the restart criteria holds or all termination criteria hold \vspace{3pt}
                 \STATE{\textbf{restart the inner loop:}
                 $\tau_{r}=t, w^{r+1,0}=\bw^{r,\tau_{r}}$ \vspace{3pt}},
                  \STATE{$\displaystyle \sigma_{r+1} ={\rm \textbf{SigmaUpdate} \vspace{3pt}}(\bw^{r,\tau_{r}},w^{r,0},\cT_1,A)$, $r=r+1;$ \vspace{3pt}}
             \UNTIL   termination criteria hold \vspace{3pt}
             \STATE{\textbf{Output:} $\{\bw^{r,t}\}$.}
		\end{algorithmic}
	\end{algorithm}

\subsubsection{Restart strategy}
Restarting has been recognized as particularly important for Halpern iterations. As noted in Theorem~\ref{Th:complexity-acc-pADMM}, the complexity bound depends on the weighted distance $R_0$ between the initial point and the optimal solution. Consequently, as the iterates approach optimality, continuing to reference a distant initial anchor becomes counterproductive, whereas resetting the anchor to the current iterate helps reduce the bound and refocus the iteration near the solution. This observation motivates the merit function
\[
R_{r,t} := \|w^{r,t} - w^*\|_{\cM}, \quad \forall r \geq 0, \ t \geq 0,
\]
where \(w^*\) is any solution of the KKT system~\eqref{eq:KKT}. Since \(w^*\) is unknown, the practical surrogate
\[
\widetilde{R}_{r,t} := \|w^{r,t} - \hat{w}^{r,t+1}\|_{\cM}
\]
is employed in defining restart rules. The following criteria are commonly adopted:
\begin{enumerate}
    \item \textbf{Sufficient decay:}
    \begin{equation}\label{eq:restart_1}
        \widetilde{R}_{r,t+1}\leq \alpha_{1}\widetilde{R}_{r,0};
    \end{equation}
    \item \textbf{Necessary decay + no local progress:}
    \begin{equation}\label{eq:restart_2}
        \widetilde{R}_{r,t+1}\leq \alpha_{2}\widetilde{R}_{r,0}, 
        \quad \text{and}\quad \widetilde{R}_{r,t+1} > \widetilde{R}_{r,t};
    \end{equation}
    \item \textbf{Long inner loop:}
    \begin{equation}\label{eq:restart_3}
        t \geq \alpha_3 k;
    \end{equation}
\end{enumerate}
where \(0<\alpha_1<\alpha_2<1\) and \(0<\alpha_3<1\). When any criterion is met, the inner loop is restarted at iteration \((r+1)\) with \(w^{r+1,0} = \bar{w}^{r,\tau_r}\) and an updated \(\sigma_{r+1}\).

\begin{remark}
Restart strategies are commonly used in first-order methods for LP~\cite{applegate2021practical,lu2023cupdlp,lu2023cupdlp_c,lu2024restarted}. For instance, PDLP adopts a normalized duality gap as the merit function~\cite{applegate2021practical}, while subsequent works by Lu et al.~\cite{lu2023cupdlp,lu2023cupdlp_c} introduced variants based on weighted KKT residuals.
\end{remark}

\subsubsection{Update rules for \texorpdfstring{$\sigma$}{Sigma}}
Another important enhancement of HPR methods concerns the update of the penalty parameter~$\sigma$. The update strategy is motivated by the complexity results of the HPR method in Algorithm \ref{alg:sp-HPR} (see Theorem~\ref{Th:complexity-acc-pADMM}). At a high level, the goal is to select $\sigma$ at each restart to tighten the complexity bound and thereby reduce the KKT residuals in subsequent iterations. Specifically, the ideal update is defined as the minimizer of the weighted distance to the optimal solution:
\begin{equation}\label{eq:sigma-0}
\sigma_{r+1} := \arg \min_{\sigma} \left\| w^{r+1,0} - w^* \right\|_{\mathcal{M}}^2,
\end{equation}
where $w^*$ is any solution of the KKT system~\eqref{eq:KKT}. Substituting the definition of $\mathcal{M}$ from~\eqref{def:M} leads to the closed-form expression
\begin{equation}\label{eq:sigma}
\sigma_{r+1} = \sqrt{ \frac{\| x^{r+1,0} - x^* \|^2}{\| y^{r+1,0} - y^* \|_{\mathcal{T}_1}^2 + \| A^*(y^{r+1,0} - y^*) \|^2} }.
\end{equation}
Since the optimal solution $(x^*,y^*)$ is unknown, practical implementations approximate these terms using the observed progress within each outer loop:
\begin{equation}\label{eq:Dx&Dy}
\Delta_x := \| \bx^{r,\tau_r} - x^{r,0} \|, 
\qquad 
\Delta_y := \sqrt{\| \by^{r,\tau_r} - y^{r,0} \|_{\mathcal{T}_1}^2 + \| A^*(\by^{r,\tau_r} - y^{r,0}) \|^2},
\end{equation}
which yields the implementable update rule
\begin{equation}\label{eq:sigma_approx}
\sigma_{r+1} = \frac{\Delta_x}{\Delta_y}.
\end{equation}

Several special cases of $\mathcal{T}_1$ have been investigated in the literature~\cite{chen2024hpr}:  
\begin{enumerate}
    \item \textbf{Case $\mathcal{T}_1 = 0$.}  
    This case occurs when $b:=l_c=u_c$, which arises in applications with special structure in $A$, such as optimal transport~\cite{zhang2025hot} and Wasserstein barycenter problem~\cite{zhang2022efficient}. The $y$-update then reduces to solving the linear system
\begin{equation}\label{eq:AATy=R}
AA^{*}\by^{r,t+1} = \frac{1}{\sigma_r}\big(b - A(\bx^{r,t+1} + \sigma_r (\bz^{r,t+1} - c))\big),
\end{equation}
which is computationally affordable in practice. In this case, the update rule~\eqref{eq:sigma_approx} simplifies to
\begin{equation}\label{eq:sigma_T1=0}
\sigma_{r+1} = \frac{\|\bx^{r,\tau_r} - x^{r,0}\|}{\|A^{*}(\by^{r,\tau_r} - y^{r,0})\|}.
\end{equation}
    \item \textbf{Case $\mathcal{T}_1 = \lambda_A I_{m} - AA^{*}$ with $\lambda_A\geq \|A\|^2$.}  
    Proposed in~\cite{esser2010general,chambolle2011first,xu2011class}, this choice applies when $l_c\neq u_c$ or when solving~\eqref{eq:AATy=R} directly is expensive. The $y$-update takes the form
    \begin{equation}\label{update-y}
    \by^{r,t+1}=\frac{1}{\sigma_r \lambda_A} \Big( \Pi_{\mathcal{K}} (R_{y}) - R_{y} \Big),
    \end{equation}
    where $R_{y}:= A(2 \bx^{r,t+1} - x^{r,t}) - \sigma_r \lambda_A y^{r,t}$. In this setting, the update for $\sigma$ becomes
    \begin{equation}\label{eq:sigma_T1!=0}
    \sigma_{r+1}=\frac{1}{\sqrt{\lambda_A}}\frac{\|\bx^{r,\tau_r} - x^{r,0}\|}{\|\by^{r,\tau_r} - y^{r,0}\|}.
    \end{equation}
\end{enumerate}

\begin{remark}
The update formula~\eqref{eq:sigma_T1!=0} is closely related to the primal weight update in PDLP~\cite[Algorithm~3]{applegate2021practical}, differing mainly by the presence of the factor $\lambda_A$.
\end{remark}

It is worth noting that the approximations $\Delta_x$ and $\Delta_y$ may deviate significantly from the true quantities. To address this, various smoothing schemes~\cite{applegate2021practical,lu2023cupdlp,lu2023cupdlp_c,chen2025hpr} and safeguards~\cite{chen2024hpr} have been proposed to stabilize the update. Moreover, the rule~\eqref{eq:sigma_approx} is not specific to LP; it extends directly to the HPR method for more general convex optimization problems~\cite{sun2025accelerating}, including CCQP~\cite{chen2025hpr}.

\section{Relationships among GPU-accelerated LP solvers}\label{sec:3}
Fig.~\ref{fig:connections} summarizes the connections among several popular FOMs for solving LP. Two relationships are of particular interest in this paper.  First, we demonstrate that the base algorithm of cuPDLPx ($r^2$HPDHG)~\cite{lu2025cupdlpx} is a special case of HPR-LP under suitable parameter choices.  Second, we examine the relationship between HPR-LP and EPR-LP, focusing on the interplay between the Halpern iteration and the ergodic iteration.  The following subsections elaborate on these two relationships in detail.

\begin{landscape}
\begin{figure}[p]
    \centering
    \includegraphics[width=1\linewidth]{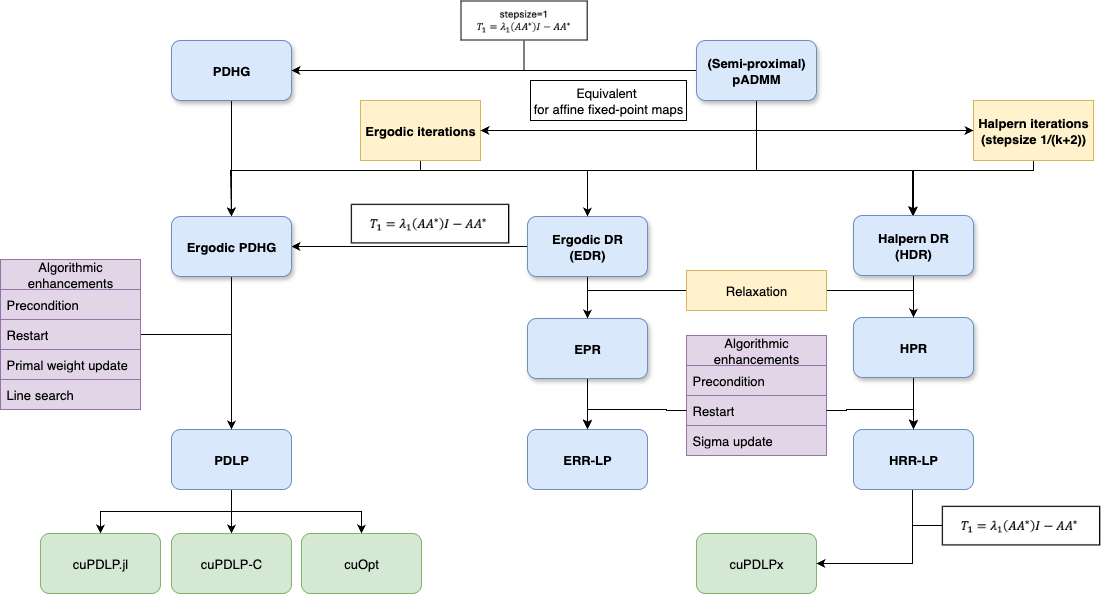} 
    \caption{Relationships among GPU-accelerated FOMs for solving LP.}
    \label{fig:connections}
\end{figure}
\end{landscape}

\subsection{The relationship between cuPDLPx and HPR-LP}
Lu et al.~\cite{lu2025cupdlpx} recently released cuPDLPx, a GPU-accelerated solver that enhances PDLP~\cite{applegate2021practical} by incorporating the Halpern iteration \cite{halpern1967fixed,lieder2021convergence}. The key observation is that the base iteration of cuPDLPx ($r^2$HPDHG) is a special case of the HPR method (Algorithm~\ref{alg:sp-HPR}). Specifically, starting from an initial point $u^0 = (y^0, x^0) \in \mathbb{R}^m \times \mathbb{R}^n$, the base iteration of cuPDLPx with reflection parameter $\gamma \in [0,1]$ is given by
\begin{equation}\label{cuPDLPx}
\begin{aligned}
\begin{cases}
    &\operatorname{PDHG}(u^k) = 
    \begin{cases}
        \bx^{k+1} = \Pi_{\mathcal{C}}\!\left(x^k - \frac{\eta}{\omega} (c - A^{*} y^k) \right), \\[6pt]
        \by^{k+1} = y^k - \eta \omega A (2\bx^{k+1} - x^k) \\[4pt]
        \quad\;\; - \eta \omega \Pi_{-\mathcal{K}}\!\left( 
            \tfrac{1}{\eta \omega} y^k - A (2\bx^{k+1} - x^k) 
        \right),
    \end{cases} \\[10pt]
    &u^{k+1} = \frac{k+1}{k+2} \Big( (1+\gamma)\, \operatorname{PDHG}(u^k) - \gamma u^k \Big) 
    + \frac{1}{k+2} u^0,
\end{cases}
\end{aligned}
\end{equation}
where $\eta$ is the stepsize and $\omega$ is the primal weight. The relationship between cuPDLPx and the HPR method is stated below.

\begin{proposition}\label{prop:equiv-cuPDLPx-HPR}
The sequence $\{(x^k,y^k)\}$ generated by cuPDLPx in~\eqref{cuPDLPx} with $\gamma=1$ coincides with the sequence produced by the HPR method (Algorithm~\ref{alg:sp-HPR}) with $\cT_1= \lambda_A I_m - AA^{*}$, provided that
\[
\sigma = \frac{\eta}{\omega}, 
\qquad 
\lambda_A = \frac{1}{\eta^2} \geq \|A\|^2,
\]
and both algorithms are initialized at the same point.
\end{proposition}
\begin{proof}
Taking $\mathcal{T}_1 = \lambda_A I_m - A A^{*}$ with $\lambda_A \geq \| A\|^2_2$, Algorithm \ref{alg:sp-HPR} simplifies to:
\begin{equation}\label{HPR-LP-2}
\begin{cases}
\begin{aligned}
& \bz^{k+1} = \frac{1}{\sigma} \left( \Pi_{\mathcal{C}} \left( x^k + \sigma (A^{*} y^k - c) \right) - \left( x^k + \sigma (A^{*} y^k - c) \right) \right), \\[6pt]
& \bx^{k+1} = x^k + \sigma (A^{*} y^k + \bz^{k+1} - c) 
= \Pi_{\mathcal{C}} \left( x^k + \sigma (A^{*} y^k - c) \right), \\[6pt]
& \by^{k+1} = \frac{1}{\sigma \lambda_A} \left( \Pi_{\mathcal{K}} \left( A(2 \bx^{k+1} - x^k) - \sigma \lambda_A y^k \right) 
- \left( A(2 \bx^{k+1} - x^k) - \sigma \lambda_A y^k \right) \right), \\[6pt]
& \hw^{k+1} = 2\bw^{k+1} - w^k, \\[6pt]
& w^{k+1} = \frac{1}{k+2} w^0 + \frac{k+1}{k+2} \hw^{k+1}.
\end{aligned}
\end{cases}
\end{equation}
Comparing the update schemes~\eqref{cuPDLPx} with $\gamma=1$ and~\eqref{HPR-LP-2}, we observe that the two methods produce identical iterates when the parameters satisfy the relations
\[
\frac{\eta}{\omega} = \sigma, \qquad \eta \omega = \frac{1}{\lambda_A \sigma},
\]
and both are initialized from the same point. That is, if 
\begin{equation}\label{eq:sigma&lambda_A}
\sigma=\frac{\eta}{\omega}, \qquad \lambda_A=\frac{1}{\eta^2}\geq \|A\|^2,    
\end{equation}
then the sequence $\{ u^k = (y^k, x^k) \}$ generated by~\eqref{cuPDLPx} with $\gamma = 1$ coincides with the sequence $\{ (y^k, x^k) \}$ produced by~\eqref{HPR-LP-2}.
\end{proof}

\begin{remark}
In~\cite{lu2025cupdlpx}, the stepsize is chosen as $\eta= 0.998/\|A\|$. By the above equivalence, one may instead take $\eta = 1/\|A\|$ to obtain a larger stepsize, thereby improving performance while preserving theoretical guarantees established in \cite{sun2025accelerating}. Furthermore, for $\gamma \in [0,1)$, the scheme~\eqref{cuPDLPx} can be viewed as a special case of the Halpern-accelerated pADMM with parameter $\rho = \gamma + 1$~\cite{sun2025accelerating}. 
\end{remark}

\subsection{The relationship between HPR-LP and EPR-LP}

In this subsection, we establish the connection between the base algorithms of HPR-LP and EPR-LP through the relationship between ergodic and Halpern iterations. To this end, we first recall the EPR method proposed in~\cite{chen2025peaceman} for solving the dual problem~\eqref{model:dualLP}, summarized in Algorithm~\ref{alg:sp-EPR}.
\begin{algorithm}[H]
	\caption{An EPR method for the LP problem~\eqref{model:dualLP}~\cite{chen2025peaceman}}
	\label{alg:sp-EPR}
	\begin{algorithmic}[1]
		\STATE \textbf{Input:} Set $\sigma > 0$ and choose a self-adjoint positive semidefinite operator $\mathcal{T}_1: \mathbb{R}^m \to \mathbb{R}^m$ such that $\mathcal{T}_1 + AA^{*}$ is positive definite. Let $w=(y,z,x)$, $\bw=(\by,\bz,\bx)$, and initialize $w^0 = (y^0, z^0, x^0)\in \mathbb{R}^m\times \mathbb{R}^n\times \mathbb{R}^n$.
        \FOR{$k=0,1,\ldots$}
		\STATE Step 1: $\displaystyle \bz^{k+1}=\arg \min_{z \in \mathbb{R}^{n}} L_{\sigma}(y^k, z ; x^k)$;
		\STATE Step 2: $\displaystyle \bx^{k+1}={x}^k+\sigma (A^{*}{y}^{k}+\bz^{k+1}-c)$;
		\STATE Step 3: $\displaystyle \by^{k+1}=\arg \min_{y \in \mathbb{R}^{m}}\{L_{\sigma}(y, \bz^{k+1} ; \bx^{k+1})+\tfrac{\sigma}{2}\|y-y^{k}\|_{\mathcal{T}_1}^2\}$;
		\STATE Step 4: $\displaystyle w^{k+1}= 2\bw^{k+1}-{w}^{k}$;
        \ENDFOR
        \STATE \textbf{Output:} $\{\bw_a^{k+1}=\tfrac{1}{k+2}\sum_{j=0}^{k+1}\bw^{j+1}\}$ and  $\{w^{k+1}_{a}=\tfrac{1}{k+2}\sum_{j=0}^{k+1}w^{j+1}\}$.
	\end{algorithmic}
\end{algorithm}
\begin{remark}
Chen et al.~\cite{chen2025peaceman} proved that the ergodic sequence $\{\bw_a^{k+1}\}$ generated by the EPR method converges for convex optimization problems with linear constraints. They also provided an example showing that the ergodic sequence $\{w^{k+1}_{a}\}$ fails to converge while the sequence $\{w^k\}$ is unbounded. 
\end{remark}

We now formalize the identity that the Halpern iteration with stepsize $1/(k+2)$ coincides with uniform ergodic averages of Picard iterates when the underlying mapping is affine. This observation goes back to the analysis of Halpern iteration~\cite{wittmann1992approximation} and has been made explicit in, e.g.,~\cite{kornlein2015quantitative}.

\begin{proposition}\label{prop:ExH-affine}
Let $\mathcal{F}:\mathbb{W}\to\mathbb{W}$ be an affine map of the form
\[
\mathcal{F}(w) \;=\; Rw + b,
\]
where $R:\mathbb{W}\to\mathbb{W}$ is linear and $b\in\mathbb{W}$ is fixed. Given $w^{0}\in\mathbb{W}$, consider the Halpern iteration
\[
    w^{k+1} \;=\; \frac{1}{k+2}\, w^{0} \;+\; \frac{k+1}{k+2}\, \mathcal{F}( w^{k})
    \qquad \forall k \ge 0.
\]
Then, for every $k\ge 0$,
\[
    w^{k+1} \;=\; \frac{1}{k+2}\sum_{j=0}^{k+1} \mathcal{F}^{\,j}(w^{0}),
\]
i.e., $w^{k+1}$ equals the ergodic average of the affine Picard iterates $\{\mathcal{F}^{\,j}(w^{0})\}_{j=0}^{k+1}$.
\end{proposition}

\begin{proof}
Let $p_j := \mathcal{F}^{\,j}(w^{0})$ for $j\ge 0$. We prove by induction on $k$ that
\begin{equation}\label{eq:ExH-aff-avg}
w^{k+1} \;=\; \frac{1}{k+2}\sum_{j=0}^{k+1} p_j .
\end{equation}
For case  $k=0$, we have $w^{1}=\tfrac12 w^{0}+\tfrac12 \mathcal{F}(w^{0})=\tfrac12(p_0+p_1)$, which is \eqref{eq:ExH-aff-avg}. Assume \eqref{eq:ExH-aff-avg} holds for some $k\ge 0$, i.e.,
\[
w^{k} \;=\; \frac{1}{k+1}\sum_{j=0}^{k} p_j .
\]
Using the Halpern update and the affinity of $\mathcal{F}$ on convex combinations,
\[
\begin{aligned}
w^{k+1}
&= \frac{1}{k+2} w^{0} + \frac{k+1}{k+2}\, \mathcal{F}( w^{k})
= \frac{1}{k+2} p_0 + \frac{k+1}{k+2}\,\mathcal{F}\!\Big(\frac{1}{k+1}\sum_{j=0}^{k} p_j\Big) \\
&= \frac{1}{k+2} p_0 + \frac{k+1}{k+2}\,\Big(\frac{1}{k+1}\sum_{j=0}^{k} \mathcal{F}(p_j)\Big)
= \frac{1}{k+2}\Big( p_0 + \sum_{j=0}^{k} p_{j+1}\Big)
= \frac{1}{k+2}\sum_{j=0}^{k+1} p_j ,
\end{aligned}
\]
which is \eqref{eq:ExH-aff-avg}. Hence, the identity holds for all $k\geq 0$.
\end{proof}

We next state a sufficient condition under which the PR updates for LP in Algorithm \ref{alg:sp-HPR} reduce to an affine fixed-point map. We adopt the following strict complementarity notion for a KKT point $(x^*,y^*,z^*)$: for each $i$,
\begin{equation}\label{eq:SC-1}
    x^*_i\in\{(l_v)_i,(u_v)_i\}\ \Rightarrow\ |z^*_i|>0,
\qquad
x^*_i\in((l_v)_i,(u_v)_i)\ \Rightarrow\ z^*_i=0,
\end{equation}
and for each $j$,
\begin{equation}\label{eq:SC-2}
(Ax^*)_j\in\{(l_c)_j,(u_c)_j\}\ \Rightarrow\ |y^*_j|>0,
\qquad
(Ax^*)_j\in((l_c)_j,(u_c)_j)\ \Rightarrow\ y^*_j=0.
\end{equation}
The associated optimal active index sets are defined as
\[
I_C^*:=\{\,i:\ x^*_i\in\{(l_v)_i,(u_v)_i\}\,\},\qquad
I_K^*:=\{\,j:\ (Ax^*)_j\in\{(l_c)_j,(u_c)_j\}\,\}.
\]
In the $\bx$–update of Algorithm~\ref{alg:sp-HPR}, we set for any $k\geq 0$,
\begin{equation}\label{def:xi}
   \xi^k:=x^k+\sigma(A^*y^k-c), 
\qquad
\bx^{k+1}=\Pi_{\mathcal{C}}(\xi^k), 
\end{equation}
and define the projection–active indices
\[
I_C^k:=\{\,i:\ [\Pi_{\mathcal{C}}(\xi^k)]_i\in\{(l_v)_i,(u_v)_i\}\,\} \quad \forall k \geq 0.
\]
Similarly, in the $\by$–update of Algorithm~\ref{alg:sp-HPR} with $\cT_1=\lambda_A I_m-AA^*$ and $\lambda_A\geq \|A\|^2$, for any $k\geq 0$, we let 
\begin{equation}\label{def:zeta}
\zeta^k:=A(2\bx^{k+1}-x^k)-\sigma\lambda_A y^k, 
\qquad
\by^{k+1}=\tfrac{1}{\sigma\lambda_A}\bigl(\Pi_{\mathcal{K}}(\zeta^k)-\zeta^k\bigr),
\end{equation}
and define
\[
I_K^k:=\{\,j:\ [\Pi_{\mathcal{K}}(\zeta^k)]_j\in\{(l_c)_j,(u_c)_j\}\,\} \quad \forall k\geq 0.
\]

We now show that under the strict complementarity condition, the active sets can be identified in finitely many steps, after which the PR map reduces to an affine operator on the corresponding face.

\begin{proposition}[Finite identification of projection-active sets]\label{thm:act-stab-bar}  
Suppose Assumption~\ref{ass:CQ} holds. Let $\cT_1=\lambda_A I_m-AA^*$ with $\lambda_A\geq \|A\|^2$.  
Assume the sequence $\{(\bx^k,\by^k,\bz^k)\}$ generated by Algorithm~\ref{alg:sp-HPR} converges to a solution $(x^*,y^*,z^*)$ satisfying strict complementarity conditions in \eqref{eq:SC-1} and \eqref{eq:SC-2}. Then there exists $K<\infty$ such that for all $k\ge K$,  
\[
I_C^k=I_C^*, \qquad I_K^k=I_K^*.
\]  
\end{proposition}

\begin{proof}
\textit{Step 1 (margins).}
Strict complementarity ensures separation from degeneracy, which implies
\[
\gamma_C:=\min_{i\in I_C^*}|z^*_i|,\quad
\gamma_K:=\min_{j\in I_K^*}|y^*_j|,\quad
\gamma:=\min\{\gamma_C,\gamma_K\}>0.
\]
For interior components, define
\[
\begin{cases}
   \beta_C:=\min_{i\notin I_C^*}\min\{\,x_i^*-(l_v)_i,\ (u_v)_i-x_i^*\,\}>0,\\
\beta_K:=\min_{j\notin I_K^*}\min\{\, (Ax^*)_j-(l_c)_j,\ (u_c)_j-(Ax^*)_j\,\}>0. 
\end{cases}
\]

\textit{Step 2 (identification for $\Pi_{\mathcal{C}}$).}
Using $A^*y^*+z^*=c$, set
\[
\xi^*:=x^*+\sigma(A^*y^*-c)=x^*-\sigma z^*.
\]
From \eqref{def:xi} and Algorithm \ref{alg:sp-HPR}, we have  $\xi^k=\bx^{k+1}-\sigma\bz^{k+1}\to \xi^*$.  
If $i \in I_C^*$, then:
- if $x_i^*=(l_v)_i$ with $z_i^*\geq \gamma$, we have $\xi_i^*\leq(l_v)_i-\sigma \gamma$, hence $[\Pi_{\mathcal{C}}(\xi^*)]_i=(l_v)_i$;  
- if $x_i^*=(u_v)_i$ with $z_i^*\leq -\gamma$, we have $\xi_i^*\geq(u_v)_i+\sigma\gamma$, hence $[\Pi_{\mathcal{C}}(\xi^*)]_i=(u_v)_i$.  
By continuity of $\Pi_{\mathcal{C}}$, there exists $K_{1,a}<\infty$ such that for all $k \ge K_{1,a}$, $[\Pi_{\mathcal{C}}(\xi^k)]_i$ remains at the corresponding bound, i.e., $i \in I_C^k$.  If $i \notin I_C^*$, then $x_i^*$ lies at least $\beta_C$ away from the bounds, so $[\Pi_{\mathcal{C}}(\xi^*)]_i=x_i^*$ is strictly interior. Since $\xi^k \to \xi^*$, there exists $K_{1,b}<\infty$ such that for all $k \ge K_{1,b}$, $[\Pi_{\mathcal{C}}(\xi^k)]_i$ also stays strictly interior, i.e., $i \notin I_C^k$. Taking $K_1=\max\{K_{1,a},K_{1,b}\}$ gives $I_C^k=I_C^*$ for all $k \ge K_1$.

\textit{Step 3 (identification for $\Pi_{\mathcal{K}}$).}  
Define $\zeta^*:=Ax^*-\sigma\lambda_A y^*$. From \eqref{def:zeta} and Algorithm~\ref{alg:sp-HPR}, we have  
\[
\zeta^k=A\bx^{k+1}+\sigma A\bz^{k+1}+\sigma(AA^*-\lambda_A I_{m})y^k-\sigma Ac \quad \forall k\ge 0.
\]
Since $(\bx^{k+1},\bz^{k+1})\to(x^*,z^*)$ and $\by^k\to y^*$, and Theorem~\ref{Th:complexity-acc-pADMM} ensures that $\|\by^{k+1}-y^k\|_{\lambda_A I_m - AA^*}\to 0$, it follows that $(AA^*-\lambda_A I_m)y^k \to (AA^*-\lambda_A I_m)y^*$ and hence $\zeta^k \to \zeta^*$. If $j\in I_K^*$, then $(Ax^*)_j=(l_c)_j$ with $y_j^*\geq\gamma$ implies $\zeta_j^*\leq(l_c)_j-\sigma\lambda_A\gamma$ and thus $[\Pi_{\mathcal{K}}(\zeta^*)]_j=(l_c)_j$; the upper-bound case is analogous. By continuity of $\Pi_{\cK}$, there exists $K_{2,a}<\infty$ such that $j\in I_K^k$ for all $k\ge K_{2,a}$. If $j\notin I_K^*$, then $(Ax^*)_j$ lies at least $\beta_K$ away from the bounds, so $[\Pi_{\mathcal{K}}(\zeta^*)]_j=(Ax^*)_j$ is interior; since $\zeta^k\to\zeta^*$, there exists $K_{2,b}<\infty$ such that $j\notin I_K^k$ for all $k\ge K_{2,b}$. Taking $K_2=\max\{K_{2,a},K_{2,b}\}$ yields $I_K^k=I_K^*$ for all $k\ge K_2$.

\textit{Step 4 (conclusion).}
Taking $K=\max\{K_1,K_2\}$ completes the proof.
\end{proof}

The next corollary shows that after active sets have been identified, the base algorithms of HPR-LP (Algorithm~\ref{alg:sp-HPR}) and EPR-LP (Algorithm~\ref{alg:sp-EPR}) coincide.
\begin{corollary}[HPR--EPR equivalence under fixed active sets]\label{cor:hpr-epr-equivalence}
Suppose Assumption~\ref{ass:CQ} holds and let $\cT_1=\lambda_A I_m-AA^*$ with $\lambda_A\geq \|A\|^2$.  
Assume that $I_C^k=I_C^*$ and $I_K^k=I_K^*$ for all $k\ge 0$.  
Suppose the HPR method (Algorithm~\ref{alg:sp-HPR}) and the EPR method (Algorithm~\ref{alg:sp-EPR}) use the same parameters $(\sigma,\lambda_A)$ and are initialized at the same point $w^0$. Then the sequence $\{w^{k+1}\}$ generated by the HPR method coincides with the ergodic sequence $\{w_a^{k+1}\}$ generated by the EPR method.
\end{corollary}

\begin{proof}
Since $I_C^k=I_C^*$ and $I_K^k=I_K^*$ for all $k\ge 0$, the projections $\Pi_{\cC}$ and $\Pi_{\cK}$ act on fixed faces.  
Each projection is affine on these faces, so the one-step PR map $w^k \mapsto \hat{w}^{\,k+1}$ in Algorithm~\ref{alg:sp-HPR} is affine.  
By Proposition~\ref{prop:ExH-affine}, the Halpern iteration (HPR) coincides with the ergodic average of the Picard iterates (EPR), hence the sequences coincide. 
\end{proof}

\begin{remark}
Proposition~\ref{thm:act-stab-bar} shows that under the strict complementarity condition, there exists $K<\infty$ such that $I_C^k=I_C^*$ and $I_K^k=I_K^*$ for all $k \ge K$.  If both the HPR method (Algorithm~\ref{alg:sp-HPR}) and the EPR method (Algorithm~\ref{alg:sp-EPR}) are initialized at $\bw^K$ and this initialization does not change the active sets $(I_C^*,I_K^*)$, then from that point onward the iterates evolve on the fixed affine face determined by these sets.  Thus, the assumption on active sets in Corollary~\ref{cor:hpr-epr-equivalence} can be realized after finitely many iterations by shifting the time index, which justifies the equivalence result in practice.  
\end{remark}

\begin{remark}
A related observation was made by Lu et al.~\cite{lu2024restarted}, who proved that in unconstrained bilinear problems the ergodic sequence of PDHG is identical to that of its Halpern-accelerated variant.  
\end{remark}

\section{Numerical experiments}\label{sec:4}
In this section, we present extensive numerical results on standard LP benchmark datasets. We begin with evaluating the impact of individual algorithmic components, followed by comparisons with state-of-the-art GPU-accelerated solvers.  

\subsection{Experimental setup}
\paragraph{Computing environment}  
All solvers are benchmarked on a SuperServer SYS-420GP-TNR equipped with an NVIDIA A100-SXM4-80GB GPU, an Intel Xeon Platinum 8338C CPU @ 2.60 GHz, and 256 GB RAM.

\paragraph{Datasets}  
We evaluate the solvers on two standard LP benchmark datasets.  
(i) The Mittelmann LP benchmark\footnote{\url{https://plato.asu.edu/ftp/lpfeas.html}}, which is widely used for testing both commercial and open-source LP solvers.  
(ii) A subset of 18 large-scale instances from MIPLIB 2017 \cite{gleixner2021miplib}, where the number of nonzeros in $A$ exceeds $10^7$, and we solve their LP relaxations.

\paragraph{Shifted geometric mean}
To evaluate solver performance across instances, we report the shifted geometric mean (SGM) with shift $\Delta = 10$, as in the Mittelmann benchmarks.  For a shift \(\Delta = 10\), the SGM10 is defined as
\(
\left( \prod_{i=1}^n (t_i + \Delta) \right)^{1/n} - \Delta,
\)
where \(t_i\) denotes the solve time in seconds for the \(i\)-th instance. Unsolved instances are assigned a time limit. 

\paragraph{Termination criteria} In HPR-LP, the sequence $\{\bw^{r,t}\}$ is used to check the stopping criteria. We terminate HPR-LP when the following stopping criteria are satisfied for the tolerance \(\varepsilon \in (0,\infty)\):
\begin{equation}\label{eq:stop-criteria}
 \begin{aligned}
&\left| -\delta^{*}_{\cK}(-y) -\delta^{*}_{\cC}(-z) - \langle c, x \rangle \right| \leq \varepsilon \left( 1 + \left|\delta^{*}_{\cK}(-y) + \delta^{*}_{\cC}(-z) \right| + \left| \langle c, x \rangle \right| \right), \\
&\|Ax -\Pi_{\cK}(Ax)\| \leq \varepsilon \left( 1 + \| \bar{b} \| \right), \\
&\left\| c - A^{*} y - z \right\| \leq \varepsilon \left( 1 + \| c \| \right),
\end{aligned}   
\end{equation}
where \(\bar{b} := \max(|l_c|, |u_c|)\) is taken componentwise, treating any infinite entries in \(l_c\) or \(u_c\) as zero when computing the norm. Other solvers are run with their default stopping rules, which are comparable to \eqref{eq:stop-criteria}. In Subsection~\ref{subsec:impact}, we refer to the relative KKT residual and duality gap in \eqref{eq:stop-criteria} collectively as the \emph{Optimality} to assess the effect of algorithmic enhancements.

\subsection{Impact of algorithmic enhancements}\label{subsec:impact}  
We evaluate the impact of adaptive restarts, penalty parameter updates, and the relaxation step on two representative Mittelmann LP instances (\texttt{datt256} and \texttt{nug08-3rd}).  We consider these components incrementally: starting from the baseline PR method (Algorithm~\ref{alg:sp-HPR} without the Halpern step), then HPR (Algorithm~\ref{alg:sp-HPR}), HPR with adaptive restarts, HPR with adaptive restarts plus adaptive penalties (Algorithm~\ref{alg:HPR_LP}), and finally an HDR variant (obtained by replacing Step~4 in Algorithm~\ref{alg:sp-HPR} with $\hw^{k+1} = \bw^{k+1}$), combined with the similar restart and adaptive penalty updates as in HPR-LP (Algorithm~\ref{alg:HPR_LP}). The results are reported in Fig.~\ref{fig:summary}, and the main observations are as follows:  
\begin{enumerate}
    \item The plain PR method fails to converge. In contrast, HPR without restart converges, albeit slowly. Incorporating adaptive restarts upgrades the rate to linear convergence, and adaptive penalty updates further accelerate convergence.  
      \item When both restart and adaptive penalty updates are employed, HPR consistently outperforms HDR, underscoring the critical role of the relaxation step (Step~4 in Algorithm~\ref{alg:sp-HPR}) in enhancing performance.  
\end{enumerate}

\begin{figure}[H]
\centering
\begin{subfigure}{0.48\textwidth}
    \includegraphics[width=\textwidth,height=4.5cm]{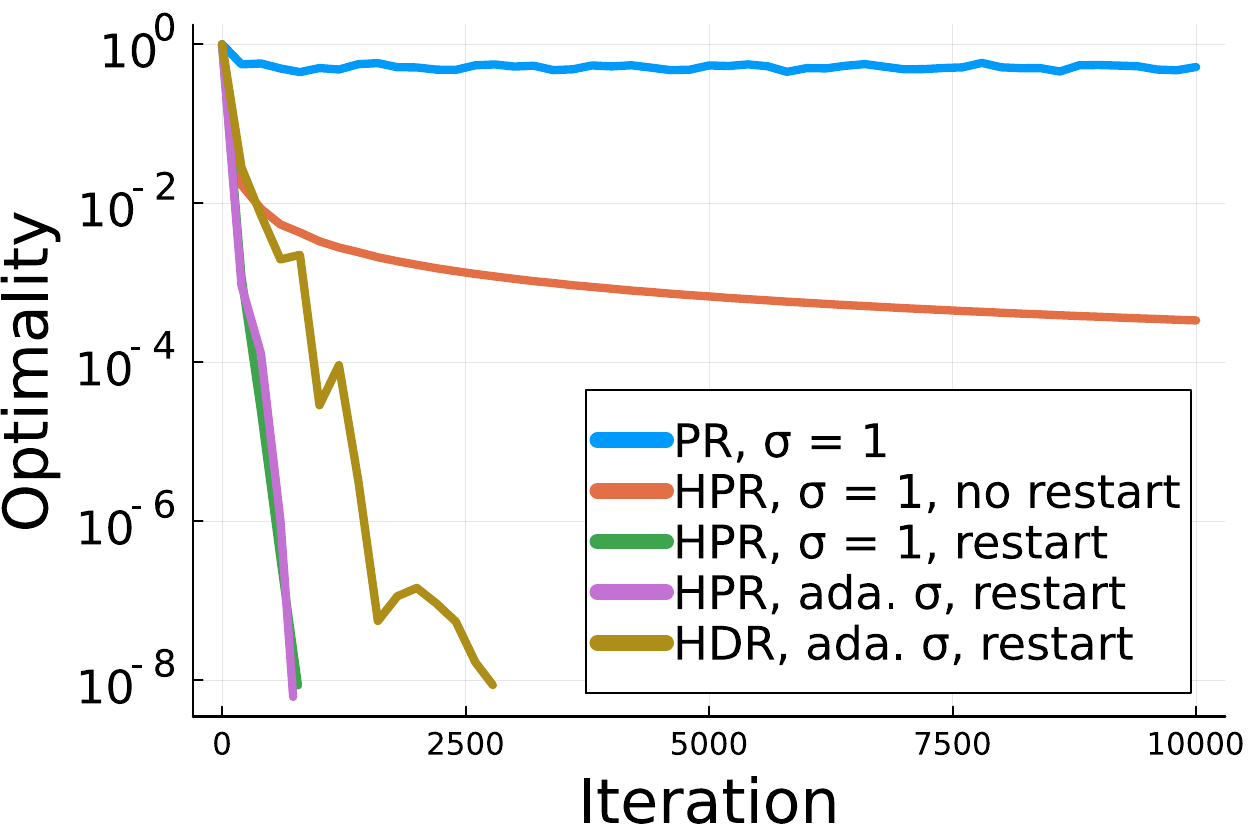}
    \caption{\texttt{datt256}}
\end{subfigure}\hfill
\begin{subfigure}{0.48\textwidth}
    \includegraphics[width=\textwidth,height=4.5cm]{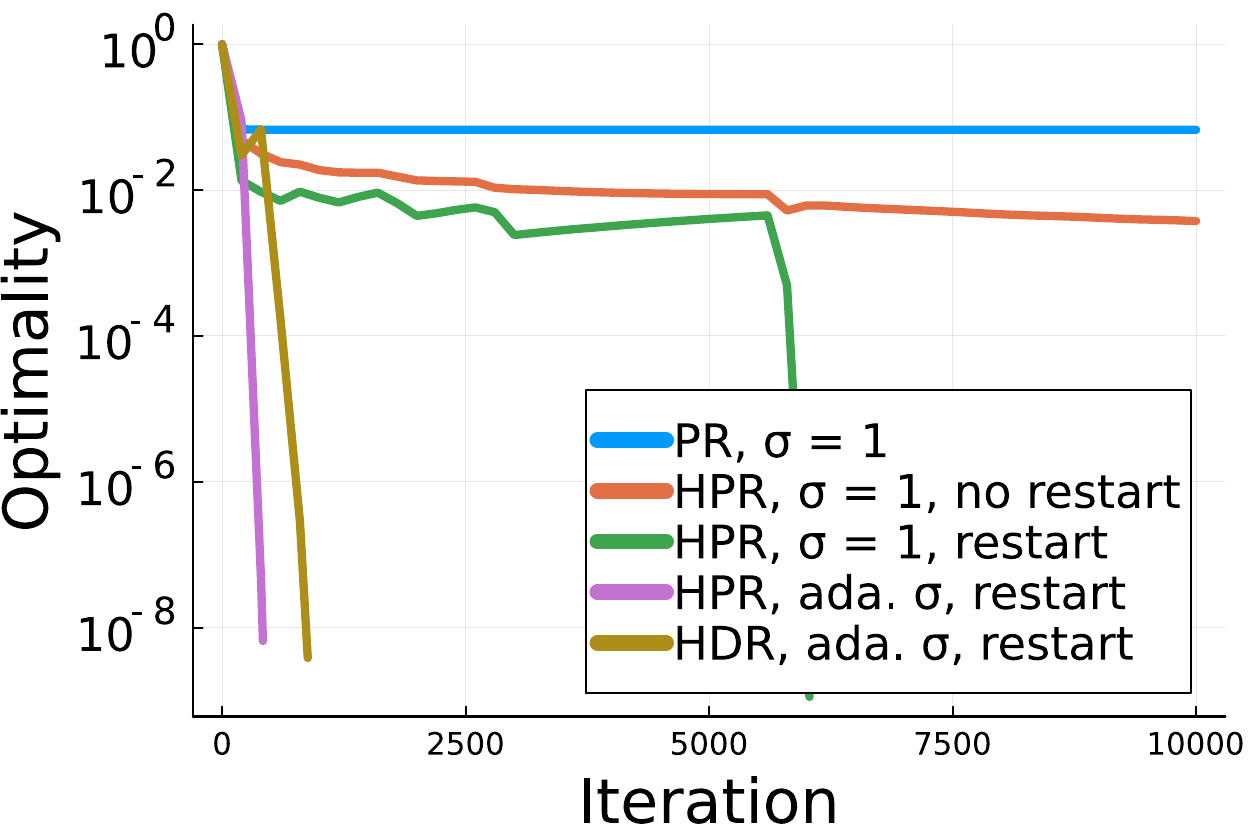}
    \caption{\texttt{nug08-3rd}}
\end{subfigure}
\caption{Performance of HPR methods with different enhancements.}
\label{fig:summary}
\end{figure}

\paragraph{Restart} We next compare the adaptive restart strategy with fixed-frequency restart. Fig.~\ref{fig:restart} reports the performance of the HPR method in Algorithm~\ref{alg:sp-HPR} (with $\sigma$ fixed at $1$) under three settings: no restart, fixed-frequency restart (with the frequency tuned for each instance), and adaptive restart as described in Subsection~\ref{subsec:algo-enhance}. When a restart strategy is applied, the HPR method exhibits linear convergence, with the rate depending on the restart schedule. Overall, adaptive restart consistently achieves slightly better performance than the best fixed-frequency choice.  

\begin{figure}[H]
\centering
\begin{subfigure}{0.48\textwidth}
    \includegraphics[width=\textwidth,height=4.5cm]{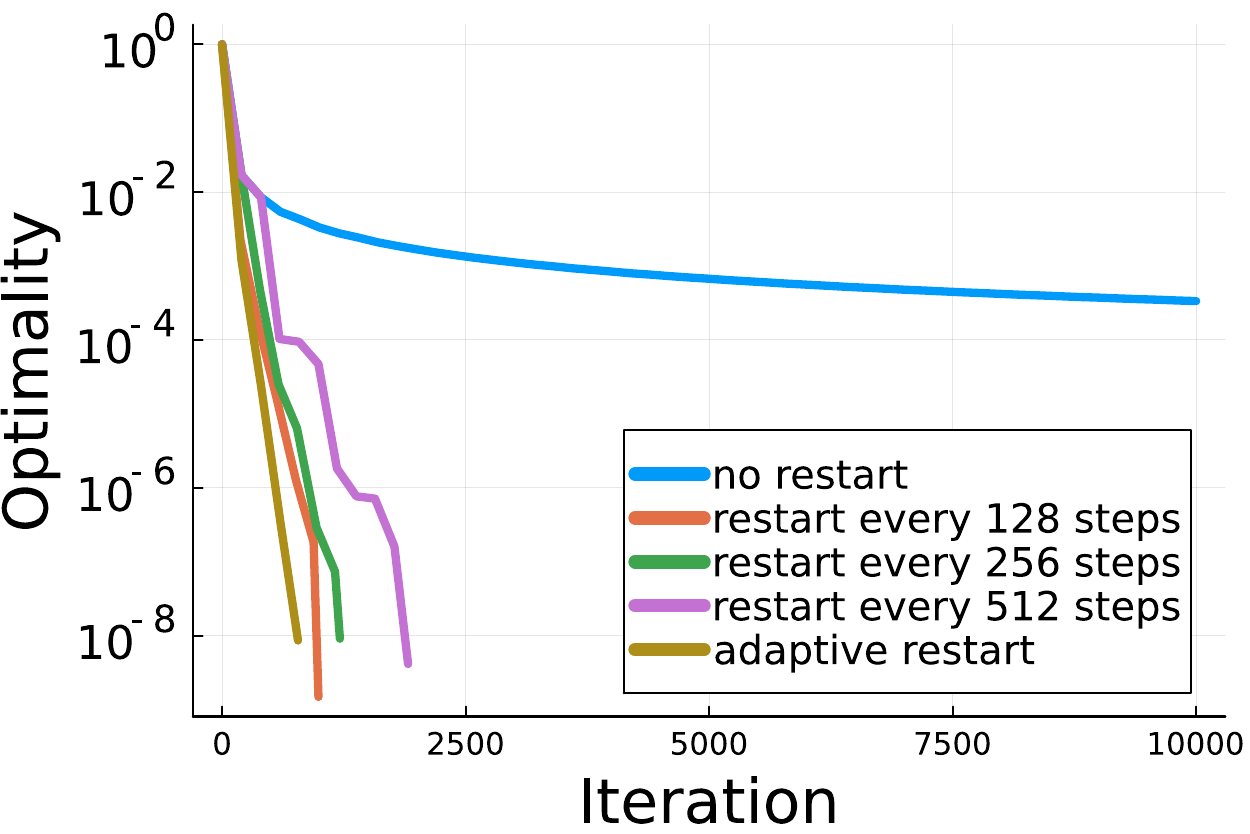}
    \caption{\texttt{datt256}}
\end{subfigure}\hfill
\begin{subfigure}{0.48\textwidth}
    \includegraphics[width=\textwidth,height=4.5cm]{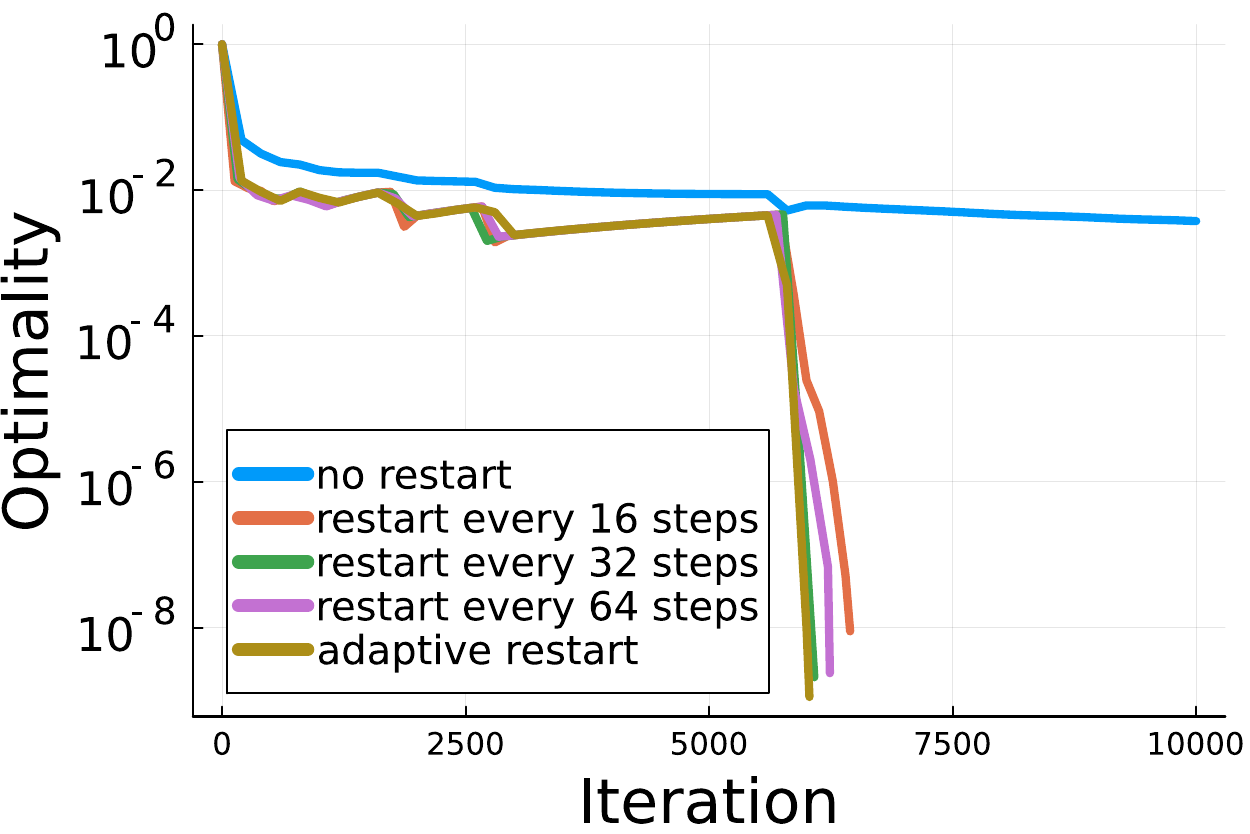}
    \caption{\texttt{nug08-3rd}}
\end{subfigure}
\caption{Performance of HPR methods under different restart strategies.}
\label{fig:restart}
\end{figure}

\paragraph{Penalty parameter $\sigma$} We next evaluate the effect of the penalty parameter in the HPR method with adaptive restart. Fig.~\ref{fig:sigma} reports the performance of HPR with different fixed values of $\sigma$ as well as with adaptive updates. We observe that for certain fixed choices of $\sigma$ (e.g., $\sigma=100$ for \texttt{datt256} and $\sigma=0.1$ or $1$ for \texttt{nug08-3rd}), the HPR method often exhibits a two-phase behavior: an initial slow progress phase (phase~I), followed by linear convergence (phase~II). The choice of $\sigma$ significantly affects the length of phase~I, while adaptive $\sigma$ selection consistently shortens this phase and yields the best overall performance.  

\begin{figure}[H]
\centering
\begin{subfigure}{0.48\textwidth}
    \includegraphics[width=\textwidth,height=4.5cm]{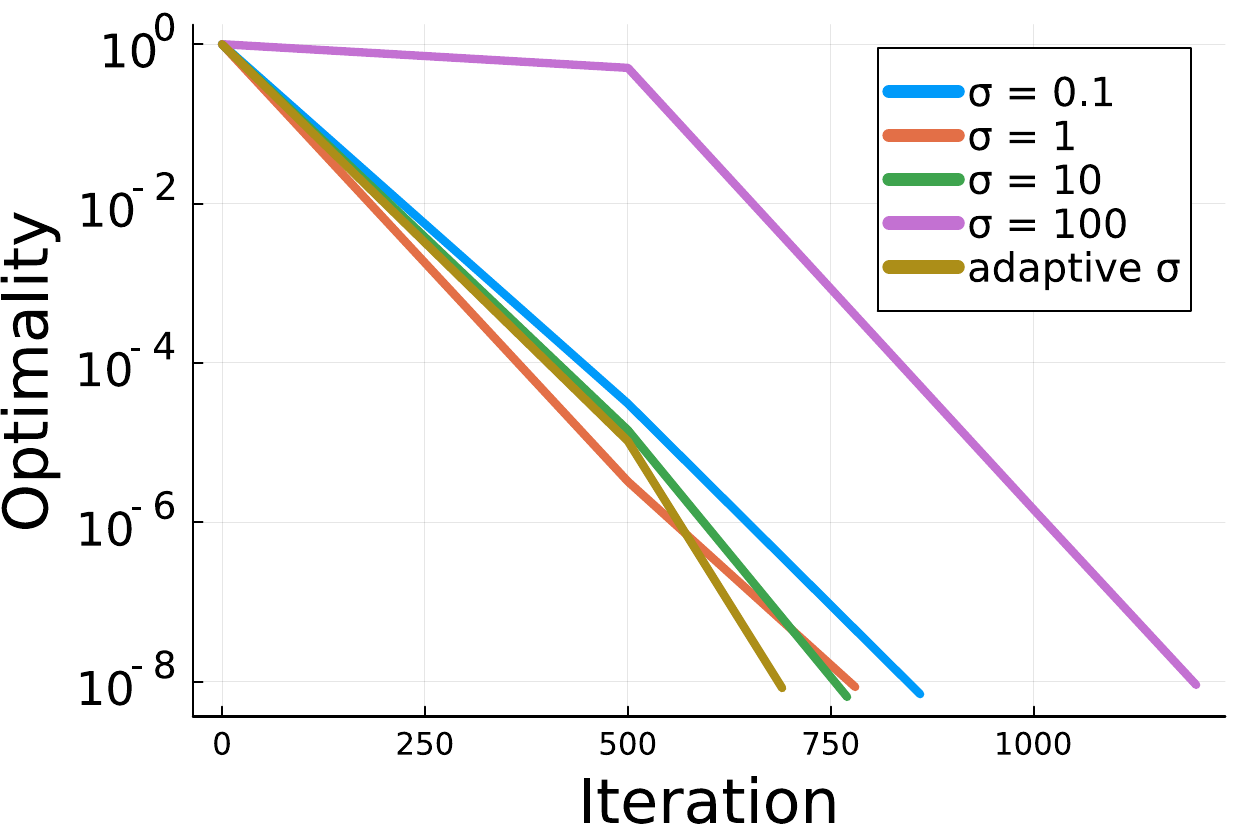}
    \caption{\texttt{datt256}}
\end{subfigure}\hfill
\begin{subfigure}{0.48\textwidth}
    \includegraphics[width=\textwidth,height=4.5cm]{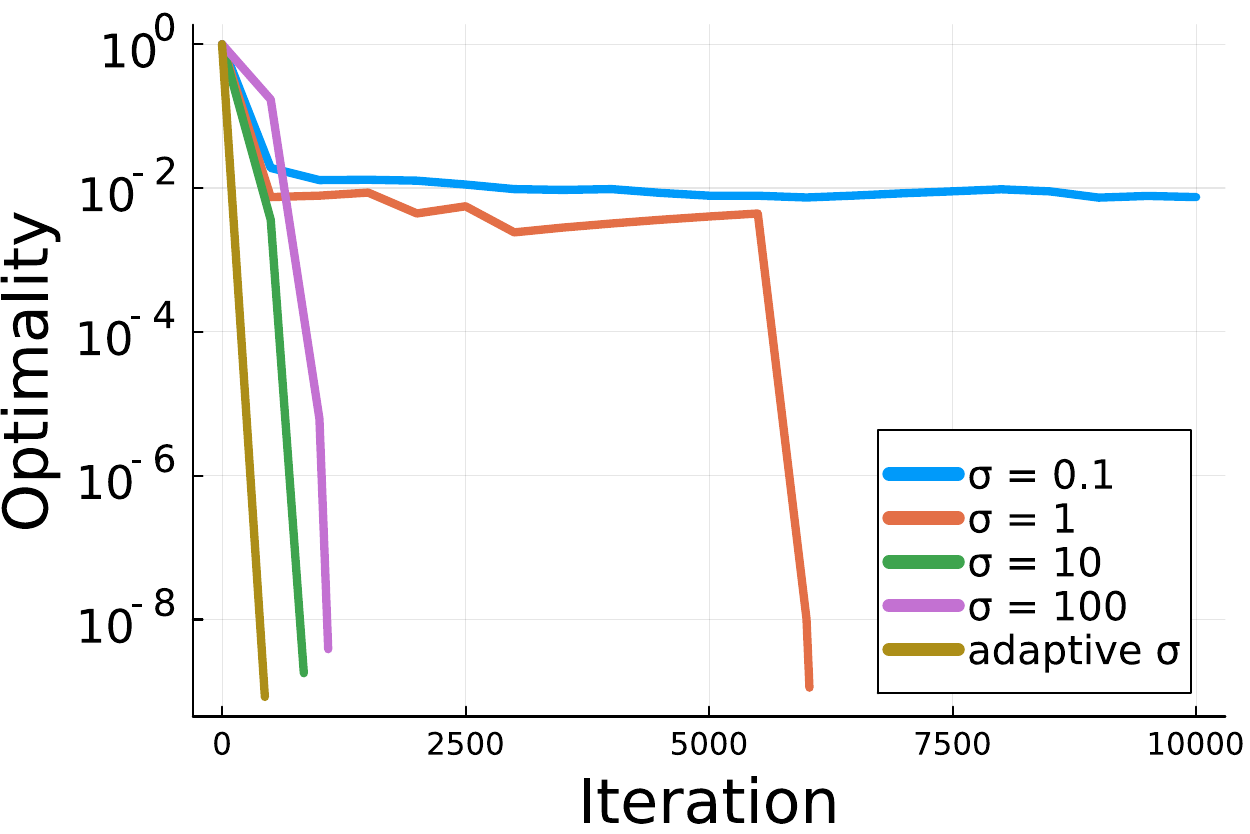}
    \caption{\texttt{nug08-3rd}}
\end{subfigure}
\caption{Performance of HPR methods under different penalty parameter strategies.}
\label{fig:sigma}
\end{figure}

\subsection{Performance comparison on benchmark datasets}
We compare the performance of several state-of-the-art solvers on two benchmark datasets, including 49 instances from the Mittelmann LP benchmark and a subset of 18 large-scale instances from MIPLIB 2017. 
The tested solvers include HPR-LP (Julia implementation)\footnote{\url{https://github.com/PolyU-IOR/HPR-LP}, v0.1.2} \cite{chen2024hpr}, EPR-LP (Julia implementation, incorporating restart and penalty parameter update techniques similar to HPR-LP) \cite{chen2025peaceman}, cuPDLPx (C implementation)\footnote{\url{https://github.com/MIT-Lu-Lab/cuPDLPx}, v0.1.0} \cite{lu2025cupdlpx},
and the GPU-accelerated PDLP family: 
cuPDLP.jl (Julia implementation)\footnote{ \url{https://github.com/jinwen-yang/cuPDLP.jl}, downloaded on July 24, 2024} \cite{lu2023cupdlp}, 
cuPDLP-C (C implementation)\footnote{\url{https://github.com/COPT-Public/cuPDLP-C}, v0.4.1}  \cite{lu2023cupdlp_c}, 
and NVIDIA cuOpt (C implementation)\footnote{\url{https://github.com/NVIDIA/cuopt}, v25.08.00}. For comparison with the commercial solver, we also include Gurobi\footnote{Version 12.0.2, academic license} \cite{gurobi}, run with its barrier method. To ensure fairness, the presolve option is disabled for all solvers.

We first compare two implementations of the HPR method for LP (Algorithm~\ref{alg:sp-HPR}): HPR-LP and cuPDLPx. The results on the Mittelmann LP benchmark and the MIPLIB 2017 relaxations are reported in Tables~\ref{tab:hans-benchmarks-HPR-LP-cuPDLPx} and~\ref{tab:MIP-HPR-LP-cuPDLPx}, respectively. On the Mittelmann benchmark, cuPDLPx performs comparably to HPR-LP, with both solving 44/49 instances, while HPR-LP is about 1.1$\times$ faster in terms of SGM10. On the MIPLIB relaxations, HPR-LP shows a clearer advantage, solving two additional instances and achieving a \textbf{1.8$\times$} speedup. These results confirm that both solvers are effective realizations of the HPR method, with our HPR-LP implementation exhibiting greater robustness and efficiency. Therefore, we do not include cuPDLPx in subsequent comparisons with other solvers.

\begin{table}[H]
\centering
\caption{Comparison between HPR-LP and cuPDLPx on the Mittelmann LP benchmark (49 instances). 
Accuracy $=10^{-8}$, time limit 15{,}000s.}
\label{tab:hans-benchmarks-HPR-LP-cuPDLPx}
\renewcommand{\arraystretch}{1.15}
\begin{tabularx}{\textwidth}{l *{2}{>{\centering\arraybackslash}X}}
\hline
\textbf{Solvers} & \textbf{HPR-LP} & \textbf{cuPDLPx} \\
\hline
SGM10 (s) & 82.1  & 90.7 \\
Solved    & 44    & 44   \\
\hline
\end{tabularx}
\end{table}
\begin{table}[H]
\centering
\caption{Comparison between HPR-LP and cuPDLPx on MIPLIB 2017 LP relaxations (18 instances). 
Accuracy $=10^{-8}$, time limit 18{,}000s.}
\label{tab:MIP-HPR-LP-cuPDLPx}
\renewcommand{\arraystretch}{1.15}
\begin{tabularx}{\textwidth}{l *{2}{>{\centering\arraybackslash}X}}
\hline
\textbf{Solvers} & \textbf{HPR-LP} & \textbf{cuPDLPx} \\
\hline
SGM10 (s) & 204.2 & 368.0 \\
Solved    & 17    & 15    \\
\hline
\end{tabularx}
\end{table}

Table~\ref{tab:hans-benchmarks-summary} summarizes the SGM10 runtimes and number of solved instances for HPR-LP and other tested solvers on the Mittelmann LP benchmark (49 instances). Detailed runtimes for each instance are provided in Table~\ref{tab:hans-benchmarks}. Specifically, HPR-LP delivers the best overall performance, attaining both the lowest SGM10 and the largest number of solved instances. EPR-LP is slower than HPR-LP but still outperforms all GPU-accelerated PDLP variants.  In terms of speedup, HPR-LP is approximately \textbf{3.4$\times$} faster than cuPDLP.jl, \textbf{2.1$\times$} faster than cuPDLP-C, \textbf{2.1$\times$} faster than NVIDIA cuOpt, and \textbf{1.4$\times$} faster than Gurobi.  Moreover, HPR-LP solves more instances than cuOpt and Gurobi, underscoring both its efficiency and robustness on large-scale LP problems.

\begin{table}[H]
\centering
\caption{Performance summary of solvers on the Mittelmann LP benchmark (49 instances). Accuracy $10^{-8}$, time limit 15,000s.}
\label{tab:hans-benchmarks-summary}
\renewcommand{\arraystretch}{1.15}
\setlength{\tabcolsep}{6pt}
\resizebox{\textwidth}{!}{%
\begin{tabular}{lcccccc}
\hline
\textbf{Solvers} & \textbf{HPR-LP} & \textbf{EPR-LP} & \textbf{cuPDLP.jl} & \textbf{cuPDLP-C} & \textbf{cuOpt} & \textbf{Gurobi} \\
\hline
SGM10 (s)   & 82.1 & 109.8 & 277.9 & 176.5 & 174.5 & 119.0 \\
Solved      & 44   & 43    & 40    & 40    & 40    & 41    \\
\hline
\end{tabular}%
}
\end{table}

Table~\ref{tab:mip-benchmarks-summary} reports the SGM10 runtimes and the number of solved instances on 18 large-scale MIP relaxations from MIPLIB 2017. Detailed runtimes for each instance are provided in Table~\ref{tab:mip-benchmarks}. Specifically, HPR-LP achieves the best overall performance, with the lowest SGM10 and the largest number of solved instances (17/18).  EPR-LP ranks second: it is slower than HPR-LP but still faster than all PDLP variants and Gurobi.  In terms of speedup, HPR-LP is approximately \textbf{2.1$\times$} faster than cuPDLP.jl, \textbf{1.6$\times$} faster than cuPDLP-C, \textbf{1.4$\times$} faster than NVIDIA cuOpt, and \textbf{1.9$\times$} faster than Gurobi.  
Moreover, HPR-LP solves three more instances than cuPDLP.jl, cuPDLP-C, and Gurobi (17 vs.\ 14), demonstrating its robustness on large-scale LP relaxations of MIPs.  

\begin{table}[H]
\centering
\caption{Performance summary of solvers on 18 MIP relaxations from MIPLIB 2017. 
Accuracy $10^{-8}$, time limit 18{,}000s.}
\label{tab:mip-benchmarks-summary}
\renewcommand{\arraystretch}{1.15}
\setlength{\tabcolsep}{6pt}
\resizebox{\textwidth}{!}{%
\begin{tabular}{lcccccc}
\hline
\textbf{Solvers} & \textbf{HPR-LP} & \textbf{EPR-LP} & \textbf{cuPDLP.jl} & \textbf{cuPDLP-C} & \textbf{cuOpt} & \textbf{Gurobi} \\
\hline
SGM10 (s) & 204.2 & 289.4 & 428.6 & 321.6 & 294.9 & 396.1 \\
Solved    & 17    & 16    & 14    & 14    & 16    & 14    \\
\hline
\end{tabular}%
}
\end{table}

Overall, the numerical experiments demonstrate that HPR-LP consistently achieves the best performance across different benchmark datasets. It not only attains the lowest SGM10 runtimes but also solves the largest number of instances, confirming both its efficiency and robustness.

\section{Conclusions and future directions}\label{sec:5}
In this paper, we studied the relationships among GPU-accelerated FOMs for LP, focusing on HPR-LP and its links to other recent solvers. Specifically, we showed that the base algorithm of cuPDLPx is a special case of the base algorithm of HPR-LP and proved that, once active sets are identified, HPR-LP and EPR-LP become equivalent through the correspondence between Halpern iteration and ergodic iteration. Numerical experiments on benchmark datasets confirmed the efficiency and robustness of HPR-LP. These results motivate the HPR framework as a solid foundation for advancing GPU-accelerated solvers for LP and related optimization problems. Several promising directions remain for developing the next generation of high-performance solvers based on HPR methods for LP and beyond:
\begin{enumerate}
    \item First, recent progress highlights the broader potential of HPR methods: Chen et al.~\cite{chen2025hpr} proposed a dual HPR method for CCQP, termed HPR-QP, which significantly outperforms existing solvers such as PDQP~\cite{lu2025practical} and SCS~\cite{o2016conic,o2021operator}. Building on the connections established in this paper, one may consider combining the implementation strategies of different solvers to produce more efficient realizations of the HPR method for solving large-scale LP problems and beyond. 
    \item Second, an important theoretical direction is to establish linear convergence guarantees for restarted variants of the HPR method, as empirically observed in Fig.~\ref{fig:restart}.  In \cite{lu2024restarted}, Lu and Yang proved that the restarted $r^2$HPDHG method with $\eta<1/\|A\|$ achieves a linear rate at the restart points:
\[
\operatorname{dist}_{\cP_\eta}\!\left(u^{r+1,0}, \mathcal{U}^*\right) 
\;\leq\; \left(\tfrac{1}{e}\right)^{r+1}\, 
\operatorname{dist}_{\cP_\eta}\!\left(u^{0,0}, \mathcal{U}^*\right) \quad \forall r\geq 0,
\]
where $\mathcal{U}^*$ denotes the optimal solution set, and 
\[
\cP_\eta \;=\;
\begin{pmatrix}
\tfrac{1}{\eta} I_m & A \\[4pt] 
A^{*} & \tfrac{1}{\eta} I_n
\end{pmatrix}
\]
is positive definite whenever $\eta<1/\|A\|$. At the critical parameter $\eta = 1/\|A\|$, however, $\cP_\eta$ becomes merely positive semidefinite, so $\operatorname{dist}_{\cP_\eta}$ is no longer a valid distance and the above linear rate guarantee breaks down. Given the equivalence between the base algorithm of HPR-LP and cuPDLPx ($r^2$HPDHG) in Proposition~\ref{prop:equiv-cuPDLPx-HPR}, it would be interesting to combine this equivalence with the theoretical framework of HPR developed in \cite{sun2025accelerating} to investigate whether linear convergence guarantees can be extended to the degenerate case $\eta = 1/\|A\|$.  

\item Third, on the practical side, integrating AI and machine learning into HPR methods—for instance, using reinforcement learning to design restart rules or adapt penalty parameters~\cite{ichnowski2021accelerating}—offers a promising avenue to further enhance their robustness and efficiency.
\end{enumerate}


\bibliographystyle{abbrv}

\bibliography{ref.bib}

\begin{thebibliography}{10}

\bibitem{applegate2021practical}
D.~Applegate, M.~D{\'\i}az, O.~Hinder, H.~Lu, M.~Lubin, B.~O'Donoghue, and
  W.~Schudy.
\newblock Practical large-scale linear programming using primal-dual hybrid
  gradient.
\newblock {\em Advances in Neural Information Processing Systems},
  34:20243--20257, 2021.

\bibitem{applegate2025pdlp}
D.~Applegate, M.~D{\'\i}az, O.~Hinder, H.~Lu, M.~Lubin, B.~O'Donoghue, and
  W.~Schudy.
\newblock {PDLP}: {A} practical first-order method for large-scale linear
  programming.
\newblock {\em arXiv preprint arXiv:2501.07018}, 2025.

\bibitem{applegate2023faster}
D.~Applegate, O.~Hinder, H.~Lu, and M.~Lubin.
\newblock Faster first-order primal-dual methods for linear programming using
  restarts and sharpness.
\newblock {\em Math. Program.}, 201(1):133--184, 2023.

\bibitem{basu2020eclipse}
K.~Basu, A.~Ghoting, R.~Mazumder, and Y.~Pan.
\newblock {E}{C}{L}{I}{P}{S}{E}: {A}n extreme-scale linear program solver for
  web-applications.
\newblock In {\em International Conference on Machine Learning}, pages
  704--714. PMLR, 2020.

\bibitem{bredies2022degenerate}
K.~Bredies, E.~Chenchene, D.~A. Lorenz, and E.~Naldi.
\newblock Degenerate preconditioned proximal point algorithms.
\newblock {\em SIAM J. Optim.}, 32(3):2376--2401, 2022.

\bibitem{chambolle2011first}
A.~Chambolle and T.~Pock.
\newblock A first-order primal-dual algorithm for convex problems with
  applications to imaging.
\newblock {\em J. Math. Imaging Vision}, 40:120--145, 2011.

\bibitem{chambolle2016ergodic}
A.~Chambolle and T.~Pock.
\newblock On the ergodic convergence rates of a first-order primal--dual
  algorithm.
\newblock {\em Math. Program.}, 159(1):253--287, 2016.

\bibitem{chen2024hpr}
K.~Chen, D.~F. Sun, Y.~Yuan, G.~Zhang, and X.~Zhao.
\newblock {HPR-LP}: {A}n implementation of an {HPR} method for solving linear
  programming.
\newblock {\em arXiv preprint arXiv:2408.12179}, 2024.
\newblock Math. Program. Comput. (2025) XXX, in print.

\bibitem{chen2025hpr}
K.~Chen, D.~F. Sun, Y.~Yuan, G.~Zhang, and X.~Zhao.
\newblock {HPR-QP}: {A} dual {H}alpern {P}eaceman--{R}achford method for
  solving large-scale convex composite quadratic programming.
\newblock {\em arXiv preprint arXiv:2507.02470}, 2025.

\bibitem{chen2025peaceman}
K.~Chen, D.~F. Sun, Y.~Yuan, G.~Zhang, and X.~Zhao.
\newblock {P}eaceman--{R}achford splitting method converges ergodically for
  solving convex optimization problems.
\newblock {\em arXiv preprint arXiv:2501.07807}, 2025.

\bibitem{chen2024cuclarabel}
Y.~Chen, D.~Tse, P.~Nobel, P.~Goulart, and S.~Boyd.
\newblock {CuClarabel}: {GPU} acceleration for a conic optimization solver.
\newblock {\em arXiv preprint arXiv:2412.19027}, 2024.

\bibitem{cui2016convergence}
Y.~Cui, X.~Li, D.~F. Sun, and K.-C. Toh.
\newblock On the convergence properties of a majorized alternating direction
  method of multipliers for linearly constrained convex optimization problems
  with coupled objective functions.
\newblock {\em J. Optim. Theory Appl.}, 169:1013--1041, 2016.

\bibitem{dantzig1963linear}
G.~B. Dantzig.
\newblock {\em Linear Programming and Extensions}.
\newblock Princeton university press, 1963.

\bibitem{davis2016convergence}
D.~Davis and W.~Yin.
\newblock Convergence rate analysis of several splitting schemes.
\newblock In {\em Splitting Methods in Communication, Imaging, Science, and
  Engineering}, pages 115--163. Springer, 2016.

\bibitem{deng2024enhanced}
Q.~Deng, Q.~Feng, W.~Gao, D.~Ge, B.~Jiang, Y.~Jiang, J.~Liu, T.~Liu, C.~Xue,
  Y.~Ye, and C.~Zhang.
\newblock An enhanced alternating direction method of multipliers-based
  interior point method for linear and conic optimization.
\newblock {\em INFORMS J. Comput}, 37(2):338--359, 2025.

\bibitem{eckstein1992douglas}
J.~Eckstein and D.~P. Bertsekas.
\newblock On the {D}ouglas-{R}achford splitting method and the proximal point
  algorithm for maximal monotone operators.
\newblock {\em Math. Program.}, 55(1):293--318, 1992.

\bibitem{esser2010general}
E.~Esser, X.~Zhang, and T.~F. Chan.
\newblock A general framework for a class of first order primal-dual algorithms
  for convex optimization in imaging science.
\newblock {\em SIAM J. Imaging Sci.}, 3(4):1015--1046, 2010.

\bibitem{gabay1976dual}
D.~Gabay and B.~Mercier.
\newblock A dual algorithm for the solution of nonlinear variational problems
  via finite element approximation.
\newblock {\em Comput. Math. Appl.}, 2(1):17--40, 1976.

\bibitem{gleixner2021miplib}
A.~Gleixner, G.~Hendel, G.~Gamrath, T.~Achterberg, M.~Bastubbe, T.~Berthold,
  P.~Christophel, K.~Jarck, T.~Koch, J.~Linderoth, et~al.
\newblock {M}{I}{P}{L}{I}{B} 2017: data-driven compilation of the 6th
  mixed-integer programming library.
\newblock {\em Math. Program. Comput.}, 13(3):443--490, 2021.

\bibitem{glowinski1975approximation}
R.~Glowinski and A.~Marroco.
\newblock Sur l'approximation, par {\'e}l{\'e}ments finis d'ordre un, et la
  r{\'e}solution, par p{\'e}nalisation-dualit{\'e} d'une classe de
  probl{\`e}mes de dirichlet non lin{\'e}aires.
\newblock {\em Revue Fran{\c{c}}aise D'automatique, Informatique, Recherche
  Op{\'e}rationnelle. Analyse Num{\'e}rique}, 9(R2):41--76, 1975.

\bibitem{goulart2024clarabel}
P.~J. Goulart and Y.~Chen.
\newblock Clarabel: {An} interior-point solver for conic programs with
  quadratic objectives.
\newblock {\em arXiv preprint arXiv:2405.12762}, 2024.

\bibitem{gurobi}
{Gurobi Optimization, LLC}.
\newblock {Gurobi {O}ptimizer {R}eference {M}anual}, 2025.

\bibitem{halpern1967fixed}
B.~Halpern.
\newblock Fixed points of nonexpanding maps.
\newblock {\em Bull. Amer. Math. Soc.}, 73(6):957--961, 1967.

\bibitem{han2018linear}
D.~Han, D.~F. Sun, and L.~Zhang.
\newblock Linear rate convergence of the alternating direction method of
  multipliers for convex composite programming.
\newblock {\em Math. Oper. Res.}, 43(2):622--637, 2018.

\bibitem{huang2024restarted}
Y.~Huang, W.~Zhang, H.~Li, D.~Ge, H.~Liu, and Y.~Ye.
\newblock Restarted primal-dual hybrid conjugate gradient method for
  large-scale quadratic programming.
\newblock {\em arXiv preprint arXiv:2405.16160}, 2024.

\bibitem{cplex2009v12}
IBM.
\newblock {IBM ILOG} {CPLEX} {O}ptimization {S}tudio {CPLEX} {U}ser’s
  {M}anual, 1987.

\bibitem{ichnowski2021accelerating}
J.~Ichnowski, P.~Jain, B.~Stellato, G.~Banjac, M.~Luo, F.~Borrelli, J.~E.
  Gonzalez, I.~Stoica, and K.~Goldberg.
\newblock Accelerating quadratic optimization with reinforcement learning.
\newblock {\em Advances in Neural Information Processing Systems},
  34:21043--21055, 2021.

\bibitem{karmarkar1984new}
N.~Karmarkar.
\newblock A new polynomial-time algorithm for linear programming.
\newblock In {\em Proceedings of the sixteenth annual ACM symposium on Theory
  of computing}, pages 302--311, 1984.

\bibitem{kornlein2015quantitative}
D.~K{\"o}rnlein.
\newblock Quantitative results for {H}alpern iterations of nonexpansive
  mappings.
\newblock {\em J. Math. Anal. Appl.}, 428(2):1161--1172, 2015.

\bibitem{lieder2021convergence}
F.~Lieder.
\newblock On the convergence rate of the {H}alpern-iteration.
\newblock {\em Optim. Lett.}, 15(2):405--418, 2021.

\bibitem{lin2021admm}
T.~Lin, S.~Ma, Y.~Ye, and S.~Zhang.
\newblock An {A}{D}{M}{M}-based interior-point method for large-scale linear
  programming.
\newblock {\em Optim. Methods Softw.}, 36(2-3):389--424, 2021.

\bibitem{lin2025pdcs}
Z.~Lin, Z.~Xiong, D.~Ge, and Y.~Ye.
\newblock {PDCS}: {A} primal-dual large-scale conic programming solver with
  {GPU} enhancements.
\newblock {\em arXiv preprint arXiv:2505.00311}, 2025.

\bibitem{lions1979splitting}
P.-L. Lions and B.~Mercier.
\newblock Splitting algorithms for the sum of two nonlinear operators.
\newblock {\em SIAM J. Numer. Anal.}, 16(6):964--979, 1979.

\bibitem{lu2025cupdlpx}
H.~Lu, Z.~Peng, and J.~Yang.
\newblock {cuPDLPx}: {A} further enhanced {GPU}-based first-order solver for
  linear programming.
\newblock {\em arXiv preprint arXiv:2507.14051}, 2025.

\bibitem{lu2023cupdlp}
H.~Lu and J.~Yang.
\newblock cu{PDLP.jl}: {A} {GPU} implementation of restarted primal-dual hybrid
  gradient for linear programming in julia.
\newblock {\em arXiv preprint arXiv:2311.12180}, 2023.

\bibitem{lu2024restarted}
H.~Lu and J.~Yang.
\newblock Restarted {H}alpern {P}{D}{H}{G} for linear programming.
\newblock {\em arXiv preprint arXiv:2407.16144}, 2024.

\bibitem{lu2025overview}
H.~Lu and J.~Yang.
\newblock An overview of {GPU}-based first-order methods for linear programming
  and extensions.
\newblock {\em arXiv preprint arXiv:2506.02174}, 2025.

\bibitem{lu2025practical}
H.~Lu and J.~Yang.
\newblock A practical and optimal first-order method for large-scale convex
  quadratic programming.
\newblock {\em Math. Program.}, pages 1--38, 2025.

\bibitem{lu2023cupdlp_c}
H.~Lu, J.~Yang, H.~Hu, Q.~Huangfu, J.~Liu, T.~Liu, Y.~Ye, C.~Zhang, and D.~Ge.
\newblock cu{P}{D}{L}{P}-{C}: {A} strengthened implementation of cu{P}{D}{L}{P}
  for linear programming by {C} language.
\newblock {\em arXiv preprint arXiv:2312.14832}, 2023.

\bibitem{monteiro2013iteration}
R.~D. Monteiro and B.~F. Svaiter.
\newblock Iteration-complexity of block-decomposition algorithms and the
  alternating direction method of multipliers.
\newblock {\em SIAM J. Optim.}, 23(1):475--507, 2013.

\bibitem{o2021operator}
B.~O'Donoghue.
\newblock Operator splitting for a homogeneous embedding of the linear
  complementarity problem.
\newblock {\em SIAM J. Optim.}, 31(3):1999--2023, 2021.

\bibitem{o2016conic}
B.~O’Donoghue, E.~Chu, N.~Parikh, and S.~Boyd.
\newblock Conic optimization via operator splitting and homogeneous self-dual
  embedding.
\newblock {\em J. Optim. Theory Appl.}, 169:1042--1068, 2016.

\bibitem{rockafellar1970convex}
R.~T. Rockafellar.
\newblock {\em Convex Analysis}.
\newblock Princeton University Press, Princeton, NJ, 1970.

\bibitem{sabach2017first}
S.~Sabach and S.~Shtern.
\newblock A first order method for solving convex bilevel optimization
  problems.
\newblock {\em SIAM J. Optim.}, 27(2):640--660, 2017.

\bibitem{shen2016weighted}
L.~Shen and S.~Pan.
\newblock Weighted iteration complexity of the s{P}{A}{D}{M}{M} on the
  {K}{K}{T} residuals for convex composite optimization.
\newblock {\em arXiv preprint arXiv:1611.03167}, 2016.

\bibitem{stellato2020osqp}
B.~Stellato, G.~Banjac, P.~Goulart, A.~Bemporad, and S.~Boyd.
\newblock {O}{S}{Q}{P}: {A}n operator splitting solver for quadratic programs.
\newblock {\em Math. Program. Comput.}, 12(4):637--672, 2020.

\bibitem{sun2025accelerating}
D.~F. Sun, Y.~Yuan, G.~Zhang, and X.~Zhao.
\newblock Accelerating preconditioned {ADMM} via degenerate proximal point
  mappings.
\newblock {\em SIAM J. Optim.}, 35(2):1165--1193, 2025.

\bibitem{wittmann1992approximation}
R.~Wittmann.
\newblock Approximation of fixed points of nonexpansive mappings.
\newblock {\em Arch. Math.}, 58(5):486--491, 1992.

\bibitem{xu2011class}
M.~Xu and T.~Wu.
\newblock A class of linearized proximal alternating direction methods.
\newblock {\em J. Optim. Theory Appl.}, 151:321--337, 2011.

\bibitem{yang2025accelerated}
B.~Yang, X.~Zhao, X.~Li, and D.~F. Sun.
\newblock An accelerated proximal alternating direction method of multipliers
  for optimal decentralized control of uncertain systems.
\newblock {\em J. Optim. Theory Appl.}, 204(1):9, 2025.

\bibitem{zhangergodic}
G.~Zhang, K.~Chen, Y.~Yuan, X.~Zhao, and D.~F. Sun.
\newblock On the ergodic convergence properties of the {P}eaceman--{R}achford
  method and their applications in solving linear programming.
\newblock {\em OpenReview preprint}, 2024.

\bibitem{zhang2025hot}
G.~Zhang, Z.~Gu, Y.~Yuan, and D.~F. Sun.
\newblock H{O}{T}: {A}n efficient {H}alpern accelerating algorithm for optimal
  transport problems.
\newblock {\em IEEE Trans. Pattern Anal. Mach. Intell.}, 47(8):6703--6714,
  2025.

\bibitem{zhang2022efficient}
G.~Zhang, Y.~Yuan, and D.~F. Sun.
\newblock An efficient {H}{P}{R} algorithm for the {W}asserstein barycenter
  problem with ${O}$({D}im({P})/$\varepsilon$) computational complexity.
\newblock {\em arXiv preprint arXiv:2211.14881}, 2022.

\bibitem{zhu2008efficient}
M.~Zhu and T.~Chan.
\newblock An efficient primal-dual hybrid gradient algorithm for total
  variation image restoration.
\newblock {\em UCLA Cam Report}, 34(2), 2008.

\end{thebibliography}

\section*{Appendix}
In this appendix, we provide detailed solver results for the Mittelmann LP benchmark (49 instances) and for the MIP relaxations (18 instances) from MIPLIB 2017, reported in Table~\ref{tab:hans-benchmarks} and Table~\ref{tab:mip-benchmarks}, respectively.

\begin{table}[H]
\centering
\caption{Results on the Mittelmann LP benchmark (49 instances). Accuracy $10^{-8}$, time limit 15{,}000s. ``T'' indicates reaching the time limit, ``M'' indicates out-of-memory, and ``F'' indicates solver failure.}
\label{tab:hans-benchmarks}
\renewcommand{\arraystretch}{1.1}
\setlength{\tabcolsep}{6pt}
\resizebox{\textwidth}{!}{%
\begin{tabular}{lrrrrrrr}
\hline
\textbf{Instance} & \textbf{HPR-LP} & \textbf{EPR-LP} & \textbf{cuPDLP.jl} & \textbf{cuPDLP-C} & \textbf{cuOpt} & \textbf{Gurobi} \\
\hline
a2864               & 5.0     & 5.8     & 2.7       & 1.3      & 1.3     & F          \\
bdry2               & 3950.2  & 9367.0  & T   & T  & T & 10.7             \\
cont1               & 56.2    & 92.2    & 1871.5    & 389.2    & 361.0   & 1.0              \\
cont11              & 429.4   & 282.0   & 5373.9    & 1947.8   & 968.0   & 4.2              \\
datt256\_lp         & 0.4     & 0.4     & 2.1       & 0.2      & 0.2     & 1.1              \\
degme               & 40.2    & 45.9    & 85.8      & 60.1     & 64.7    & 17.4             \\
dlr1                & 12947.4 & T & T   & T  & T & 411.6            \\
dlr2                & T & T & T   & T  & T & 7103.3           \\
Dual2\_5000         & 72.8    & 54.7    & 9441.1    & 3310.1   & 747.6   & M          \\
ex10                & 0.5     & 0.5     & 1.5       & 0.1      & 0.2     & F          \\
fhnw-bin1           & 106.3   & 307.2   & 526.7     & 74.0     & 244.2   & 101.1            \\
FOME13              & 1.5     & 3.3     & 20.3      & 5.4      & 3.5     & 3.5              \\
graph40-40\_lp      & 0.4     & 0.3     & 1.5       & 0.1      & 0.2     & 24.9             \\
irish-electricity   & T & T & T   & T  & T & F          \\
L1\_sixm1000obs     & 13.9    & 20.4    & 38.6      & 37.5     & 74.4    & M          \\
L1\_sixm250obs      & 2.5     & 5.4     & 11.3      & 11.9     & 9.7     & 228.6            \\
L2CTA3D             & 43.3    & 49.5    & 16.3      & 9.3      & 12.0    & 542.1            \\
Linf\_five20c       & 3.7     & 30.8    & T   & T  & T & 10.7             \\
neos                & 27.9    & 83.9    & 396.0     & 199.4    & 162.5   & 212.0            \\
neos-3025225\_lp    & 5.4     & 11.5    & 9.7       & 3.9      & 3.3     & 24.1             \\
neos-5052403-cygnet & 10.6    & 10.2    & 16.3      & 8.6      & 6.7     & 11.3             \\
neos-5251015\_lp    & 1.1     & 3.8     & 3.4       & 1.2      & 1.2     & 5.7              \\
neos3               & 0.4     & 0.3     & 2.0       & 0.6      & 0.4     & 321.3            \\
ns168703            & T & T & T   & T  & T & 4.4              \\
ns168892            & T & T & T   & T  & T & F          \\
nug08-3rd           & 0.1     & 0.2     & 1.7       & 0.1      & 0.1     & 4.8              \\
pds-100             & 3.0     & 4.5     & 21.2      & 8.7      & 7.8     & 19.5             \\
physiciansched3-3   & T & T & T   & T  & T & 89.3             \\
Primal2\_1000       & 214.2   & 402.9   & 2105.8    & 1556.6   & 1069.0  & F          \\
qap15               & 1.8     & 2.7     & 14.3      & 5.4      & 2.9     & 1.2              \\
rail02              & 47.6    & 75.6    & 852.4     & 145.6    & 168.0   & 9.8              \\
rail4284            & 137.3   & 397.6   & 578.7     & 362.4    & 482.1   & 13.2             \\
rmine15\_lp         & 3.5     & 6.3     & 24.3      & 10.0     & 9.8     & 156.3            \\
s100                & 148.5   & 319.8   & 1631.4    & 571.4    & 499.4   & 6.7              \\
s250r10             & 18.3    & 37.1    & 192.6     & 105.4    & 35.6    & 3.4              \\
s82                 & 1157.1  & 3803.6  & 3390.7    & 4686.7   & 6204.9  & 75.0             \\
savsched1           & 0.2     & 0.3     & 1.9       & 0.3      & 0.2     & 6.1              \\
scpm1\_lp           & 31.7    & 53.4    & 41.9      & 27.5     & 19.6    & 9.2              \\
set-cover           & 287.5   & 383.7   & 515.3     & 290.1    & 385.5   & 60.5             \\
shs1023             & 110.5   & 144.5   & T   & T  & T & 32.8             \\
square41            & 384.3   & 326.6   & 234.9     & 138.3    & 89.3    & 79.5             \\
stat96v2            & 459.9   & 1234.5  & 1705.9    & 1290.1   & 9740.8  & F          \\
stormG2\_1000       & 10.6    & 11.9    & 41.0      & 30.5     & 30.5    & 23.1             \\
stp3d               & 15.6    & 24.9    & 157.0     & 34.4     & 27.4    & 5.3              \\
supportcase10       & 5.3     & 6.1     & 20.9      & 6.9      & 5.2     & 19.5             \\
thk\_48             & 572.2   & 489.3   & 2839.6    & 1164.5   & 755.9   & 7581.2           \\
thk\_63             & 28.0    & 47.8    & 101.4     & 49.2     & 54.3    & 270.5            \\
tpl-tub-ws1617      & 38.6    & 66.2    & 554.3     & 151.1    & 220.2   & 51.7             \\
woodlands09         & 0.7     & 0.6     & 2.3       & 0.7      & 0.5     & 7.4              \\
\hline
SGM10    & 82.1 & 109.8 & 277.9 & 176.5 & 174.5 & 119.0 \\
solved  & 44 & 43 & 40 & 40 & 40 & 41 \\
\hline
\end{tabular}%
}
\end{table}

\begin{table}[H]
\centering
\caption{Results on MIP relaxations (18 instances) from MIPLIB 2017. 
Accuracy $10^{-8}$, time limit 18{,}000s. 
``T'' indicates reaching the time limit, ``M'' indicates out-of-memory, and ``F'' indicates solver failure.}\label{tab:mip-benchmarks}
\renewcommand{\arraystretch}{1.15}
\setlength{\tabcolsep}{6pt}
\resizebox{\textwidth}{!}{%
\begin{tabular}{lrrrrrrr}
\hline
\textbf{Instance} & \textbf{HPR-LP} & \textbf{EPR-LP} & \textbf{cuPDLP.jl} & \textbf{cuPDLP-C} & \textbf{cuOpt} & \textbf{Gurobi} \\
\hline
a2864-99blp        & 4.9     & 5.8     & 8.1       & 1.3      & 1.3     & F          \\
dlr1               & 12934.8 & T & T   & T  & T & 412.6            \\
ivu06-big          & 624.5   & 905.3   & 4203.3    & 763.7    & 1347.3  & 24.8             \\
ivu59              & 929.3   & 1061.7  & 3500.5    & 368.3    & 401.7   & 30.9             \\
kottenpark09       & 80.4    & 52.0    & 60.9      & 74.6     & 59.5    & 278.2            \\
neos-2991472-kalu  & 0.4     & 0.9     & 0.5       & 0.1      & 0.1     & 91.1             \\
neos-3208254-reiu  & 35.3    & 72.3    & 43.0      & 116.0    & 170.0   & 640.6            \\
neos-3740487-motru & 1.5     & 2.4     & 1.6       & 3.4      & 3.1     & M          \\
neos-4332801-seret & 301.0   & 1229.0  & T   & T  & 9260.0  & 419.2            \\
neos-4332810-sesia & 272.8   & 785.0   & T   & T  & 2575.2  & F         \\
neos-4535459-waipa & T & T & T   & T  & T & F         \\
neos-4545615-waita & 297.6   & 394.1   & 393.4     & 2915.2   & 3069.1  & 39.3             \\
nucorsav           & 3.2     & 5.6     & 3.5       & 4.9      & 3.7     & 1862.0           \\
scpn2              & 173.5   & 121.5   & 321.8     & 44.6     & 37.8    & 17.5             \\
square41           & 383.0   & 327.7   & 328.0     & 138.0    & 89.6    & 79.0             \\
square47           & 548.1   & 697.7   & 571.8     & 133.8    & 153.8   & 57.3             \\
t11nonreg          & 1498.1  & 1683.9  & 1821.7    & 2136.0   & 2514.5  & 104.2            \\
usafa              & 30.6    & 318.0   & 33.3      & 33.2     & 66.5    & 90.6             \\
\hline
SGM10                     & 204.2   & 289.4   & 428.6     & 321.6    & 294.9   & 396.1            \\
solved                    & 17      & 16      & 14        & 14       & 16      & 14       \\
\hline
\end{tabular}%
}
\end{table}

\end{document}